\documentclass{amsart}
\usepackage{amsmath,amssymb, amsthm, tikz, mathtools}
\usepackage{mathabx}
\usepackage{ulem}
\usepackage{pgfkeys}
\usepackage{graphics}
\usepackage{enumitem}
\usepackage{hyperref}
\usepackage{textcomp}
\usepackage{comment}
\usepackage{cancel}

\newtheorem{theorem}{Theorem}[section]
\newtheorem{proposition}{Proposition}[section]
\newtheorem{lemma}{Lemma}[section]
\newtheorem{claim}{Claim}[section]
 
\newtheorem{cor}{Corollary}[section]

\newtheorem*{thm*}{ Theorem}
\newenvironment{defi}{\medskip\noindent{\sc
Definition}. }{\goodbreak\medskip}

\newenvironment{nota}{\medskip\noindent{\sc
Notation}.}{\goodbreak\medskip}
\newenvironment{remk}{\noindent{\sc
Remark}. }{\goodbreak\vskip10pt}

\newenvironment{notas}{\medskip\noindent{\sc
Notations}. }{\goodbreak\medskip}

\newenvironment{exa}{\noindent{\sc
Example}. }{\goodbreak\vskip10pt}
\newenvironment{ques}{\noindent{\sc
Question}. }{\goodbreak\vskip10pt}

\usepackage[colorinlistoftodos,french]{todonotes}

\def\cal{\mathcal}

\def\P{{\mathcal P}}
\def\M{{\mathcal M}}
\def\cb{{\mathcal B}}
\def\cq{{\mathcal Q}}
\def\cF{{\mathcal F}}

\def\cw{{\mathcal W}}
\def\cC{{\mathcal C}}

\def\cg{{\mathcal G}}

\def\cl{{\mathcal L}}
\def\cm{{\mathcal M}}
\def\cN{{\mathcal N}}
\def\cR{{\mathcal R}}
\def\cU{{\mathcal U}}
\def\cV{{\mathcal V}}
\def\cw{{\mathcal W}}

\def\cv{{\mathcal V}}

\def\cz{{\mathcal Z}}

\def\cO{{\mathcal O}}

\def\R{\mathbb{R}}

\def\Z{\mathbb{Z}}
\def\N{\mathbb{N}}

\def\C{\mathbb{C}}
\def\T{\mathbb{T}}

\def\dim{{\rm dim}\,}

\def\smallskip{\par\vspace{1mm}}
\def\medskip{\par\vspace{2mm}}
\def\bigskip{\par\vspace{3mm}}

\newcount\equationcount
\def\thenumber{0}
\def\eq#1{\global\advance\equationcount by 1
   \def\thenumber{\number\equationcount}
                        {$$#1\eqno(\thenumber)$$}}

\tikzset{
xmin/.store in=\xmin, xmin/.default=-1.5, xmin=-1.5,
xmax/.store in=\xmax, xmax/.default=7.5, xmax=7.55,
ymin/.store in=\ymin, ymin/.default=-0.75, ymin=-0.75,
ymax/.store in=\ymax, ymax/.default=3.25, ymax=3.25,
}

\makeatletter
\newcommand\footnoteref[1]{\protected@xdef\@thefnmark{\ref{#1}}\@footnotemark}
\makeatother

\begin{document}

\title[Vanishing Maslov index]{ { Vanishing asymptotic Maslov index  for  conformally  symplectic  flows}}

\author{Marie-Claude Arnaud$^{\dag,\ddag,\circ}$, Anna Florio$^{\ast,\circ}$, Valentine Roos$^{+,\circ}$}

\email{Marie-Claude.Arnaud@imj-prg.fr }
\email{florio@ceremade.dauphine.fr}
\email{valentine.roos@ens-lyon.fr}

\date{}



\thanks{$\dag$ Universit\'e   Paris Cit\'e  and Sorbonne Universit\'e, CNRS, IMJ-PRG, F-75006 Paris, France } 
\thanks{$\ddag$ Member of the {\sl Institut universitaire de France.}}
\thanks{$\circ$ ANR AAPG 2021 PRC CoSyDy: Conformally symplectic dynamics, beyond symplectic dynamics.ANR-CE40-0014}
\thanks{$\ast$ Universit\'e Paris Dauphine - Ceremade
UMR 7534
Place du M$^{\mathrm{al}}$ De Lattre De Tassigny
75775 PARIS Cedex 16, France. }
\thanks{$+$ \'Ecole Normale Sup\'erieure de Lyon, UMPA UMR 5669
46 all\'e d'Italie, 69364 Lyon Cedex 07, France.}

\begin{abstract}  {Motivated by Mather theory of minimizing measures for symplectic twist dynamics, we study conformally symplectic flows on a cotangent bundle. These dynamics} are the most general dynamics for which it makes sense to look at (asymptotic) dynamical Maslov index. {Our main result is} the existence of invariant measures with vanishing index without any convexity hypothesis, in the  general framework of conformally symplectic flows. A degenerate twist-condition hypothesis  implies the existence of \textit{ergodic} invariant measures with zero dynamical Maslov index {and thus the existence of points with zero dynamical Maslov index.}

\end{abstract}

\maketitle
 \section{Introduction and Main Results.}\label{SecIntro}

This study  mainly  concerns   conformally symplectic  flows that are defined  on the cotangent  bundle ${\cal M}=T^*M$ of a closed manifold $M$, where  ${\cal M}$ is endowed with its tautological 1-form $\lambda$, its symplectic form $\omega=-d\lambda$ and we denote by $\pi:T^*M\rightarrow M$ the usual projection.

Symplectic dynamics have been intensively studied because they model conservative phenomena, but a lot of phenomena are dissipative, e.g. mechanical systems with friction. Some of these dissipative dynamics are conformally symplectic~:  
a diffeomorphism $f :  {\cal M}\righttoleftarrow$ is conformally symplectic if for some constant $a$, we have $f^*\omega=a\omega$. When $a=1$, the diffeomorphism is symplectic.  A complete vector field $X$ on $\cal M$ is conformally symplectic if $L_X\omega=\alpha\omega$, where $L_X$ is the Lie derivative, for some $\alpha\in\R$. 

{ When $\dim M\geq 2$ and $M$ is connected, we have also the following characterization of conformally symplectic dynamics of ${\cal M}$ : a diffeomorphism $f :  {\cal M}\righttoleftarrow$ is conformally symplectic if and only if the image by $Df$ of any Lagrangian subspace in $T\cal M$ is Lagrangian. The existence of a conformal factor at every point is a result of \cite{LiveraniWojtkowski1998} and the independence of this factor from the point is a result of \cite{Libermann1959}.\\
}

 In the symplectic setting, an inspiring example  is the completely integrable Hamiltonian case. Then the manifold  is   foliated by  invariant  Lagrangian graphs. This example is of course very specific. \\
 However,  several authors  found some traces of  integrability in many non integrable cases.  Aubry-Mather  theory in the  case of exact symplectic twist maps and its vast extension  by Ma\~n\'e and Mather  to the case of Tonelli  Hamiltonian systems are such  results.\\
  In both settings, the method is variational and the  ``ghosts"  of invariant submanifolds are filled  by minimizing orbits. A cotangent bundle has a  natural  Lagrangian foliation given by its vertical fiber  and a feature of the minimizing orbits is that they have vanishing Maslov index with respect to this foliation. \\
Here,  in a more general  setting, our goal  is to prove the existence  of a large set of points with vanishing dynamical Maslov index. We  recall  that the Maslov index  ${\rm MI}(\Gamma)$ of a piece of  arc of Lagrangian subspaces $\Gamma=(\Gamma_t)_{t\in I}$ of $T{\cal M}$  is the algebraic  number of intersection of this arc with the Maslov singular cycle of the vertical foliation, i.e. the Maslov index gives more or less  the number of times when the arc is non transverse to the vertical foliation. See Subsection \ref{subsec defi MI}. The dynamical Maslov index of a Lagrangian subspace $L$ of $T{\cal M}$ for some time interval $I$ and some flow $(\phi_t)$ whose differential preserves Lagrangian subspaces, which is denoted by  ${\rm DMI}(L, (\phi_t)_{t\in I})$, is then the  Maslov index ${\rm MI}((D\phi_t(L))_{t\in I})$. The precise definitions are given in Section \ref{section def MI}.


{We begin with a preliminary statement, that is the key result  for finding invariant measures with vanishing asymptotic Maslov index.}

\begin{theorem}\label{Tppal}
 Let $\cl\subset {\cal M}$ be a Lagrangian graph.
 Let $(\phi_t)$ be an isotopy  of  conformally symplectic diffeomorphisms  of ${\cal M}$ such that $\phi_0={\rm Id}_{{\cal M}}$. Then there exists a smooth closed 1-form $\eta :M\rightarrow {\cal M}$ and a Lipschitz function $u:M\rightarrow \R$ that is $C^1$ on an open subset $U\subset M$ of full Lebesgue measure such that
$$ \forall q\in U, p:=\phi^{-1}_1(\eta(q)+du(q))\in \cl\quad\text{and}\quad {\text{\rm DMI}\Big(T_{p}\cl,(\phi_s)_{s\in[0,1]}
	\Big)=0}.$$

\end{theorem}

This theorem has important consequences concerning the so-called asymptotic Maslov index.

\begin{defi}  Let 
$(\phi_t)
$ be an isotopy  of  conformally symplectic diffeomorphisms  of ${\cal M}$ such that $\phi_0={\rm Id}_{\cal M}$. 
\begin{enumerate} 
\item Let $L\subset T{\cal M}$ be a Lagrangian subspace that is transverse to the vertical foliation. 
Whenever the limit exists, the {asymptotic Maslov index}  of $L$ for  $(\phi_t)$ is
 		\[
 		\text{DMI}_{\infty}(L, (\phi_t)
 		):=\lim_{t\to+\infty}\dfrac{\text{DMI}(L,(\phi_s)_{s\in[0,t]})}{t}.
 		\]
We will prove {\color{black}(see Corollary \ref{no subspace})} that if $L, L'\subset T_x{\cal M}$ then $$\text{DMI}_{\infty}(L, (\phi_t)
)=\text{DMI}_{\infty}(L', (\phi_t)
).$$
This allows us to introduce the following.

\item Let $(\phi_t)
$ be an isotopy  of  conformally symplectic diffeomorphisms  of ${\cal M}$ such that $\phi_0={\rm Id}_{\cal M}$. Let $x\in {\cal M}$. Then the dynamical asymptotic Maslov index at $x$ for $(\phi_t)
$ is denoted by $\text{DMI}_\infty(x, (\phi_t)
)$ and is the asymptotic Maslov index of $L$ for every Lagrangian subspace $L$ of $T_x{\cal M}$. 
\end{enumerate}
\end{defi}

The definition of (asymptotic) dynamical Maslov index first appears in the work of Ruelle \cite{Ruelle1985}: the author introduced the notion of \textit{rotation number} for surface diffeomorphisms { that are isotopic to identity} and for $3$-dimensional flows, and generalized this to symplectic dynamics. {He proves that if $(\phi_t)$ is an  isotopy such that $\phi_0={\rm Id}_{\cal M}$ and $\phi_{t+1}=\phi_t\circ \phi_1$, then for every probability measure $\mu$  invariant  by $\phi_1$ with compact support,   $\text{DMI}_\infty(x, (\phi_t)
)$ exists at $\mu$-almost every point and $x\mapsto \text{DMI}_\infty(x, (\phi_t)
)$ is a measurable and bounded function. Hence he defines
 the asymptotic Maslov index of such a measure.

\begin{defi}\label{def DMI infty measure}
Let $(\phi_t)
$ be a conformally symplectic isotopy of ${\cal M}$ such that $\phi_0=\mathrm{Id}_{\cal M}$ and $\phi_{t+1}=\phi_t\circ\phi_1$. Let $\mu$ be a $\phi_1$-invariant probability measure with compact  support. Then, the asymptotic Maslov index of $\mu$ for $(\phi_t)
$ is
\[
\mathrm{DMI}(\mu, (\phi_t)
):=\int_{\cal M}\text{DMI}_\infty(x, (\phi_t)
)\, d\mu(x).
\]
\end{defi}
If $\mu$ is a $\phi_1$-invariant \textit{ergodic} measure with compact support, then for $\mu$-almost every $x\in\M$ it holds
     \[
     \mathrm{DMI}_\infty(x,(\phi_t))=\mathrm{DMI}(\mu,(\phi_t))\, .
     \]

We will present in Proposition \ref{MI existence SCH} a proof of these results   that is a consequence of a result of Schwartzman, \cite{Schwartzman1957}.}

 
 {Our first corollary gives the existence of invariant probability measures with vanishing asymptotic Maslov index. A priori, this doesn't implies the existence of points with vanishing dynamical Maslov index.}

\begin{cor}\label{cormeasurebis}
Let $(\phi_t)
$ be a conformally symplectic isotopy of ${\cal M}$ such that $\phi_0=\mathrm{Id}_{\cal M}$ and $\phi_{t+1}=\phi_t\circ\phi_1$.
Let $\cl\subset {\cal M} $ be a Lagrangian submanifold   that is {H-isotopic}\footnote{A H-isotopy is a Hamiltonian isotopy.} to a graph and such that $\displaystyle{\bigcup_{t\in[0, +\infty)}\phi_t(\cl)}$ is relatively compact. Then there exists at least one $\phi_1$-invariant probability measure $\mu$ whose asymptotic Maslov index is zero and whose  support is in {
$$\displaystyle{\bigcap_{T\in [0, +\infty)}\overline{\bigcup_{t\in[T, +\infty)}\phi_t(\cl)}}.$$}
Moreover, if $(\phi_t)$ is a flow, then $\mu$ can be chosen $(\phi_t)$ invariant.
\end{cor}
This result applies in the autonomous conservative Tonelli case --where the Hamiltonian is a proper first integral--  or in the discounted autonomous case -- where there is a proper Lyapunov function defined in  the complement on some compact subset--.\\

As $\T^{2d}$ can be obtained as the quotient of $T^*\T^d$ by a discrete group of transformations, we obtain also a result for $\T^{2d}$. In the following statement, { the leaves of the reference Lagrangian foliation are the $d$-dimensional Lagrangian tori $\{0\}\times \T^d$. 
}

\begin{cor}\label{measuretorus}
Let $(\phi_t)
$ be a   symplectic isotopy of $\T^{2d}$ such that $\phi_0=\mathrm{Id}_{\T^{2d}}$ and $\phi_{t+1}=\phi_t\circ\phi_1$. Then, there exists at least one $\phi_1$-invariant probability measure $\mu$ whose asymptotic Maslov index is zero.
\end{cor}
In the latter corollaries, we cannot ensure that the measure is ergodic and then we don't know if there is at least one point with zero asymptotic Maslov index. Now we will give sufficient conditions to obtain such ergodic measures and such points.
{\color{black}
\begin{defi}
A Darboux  chart $F=(F_1, F_2):{\cal U}{ \,\subset\cal M} \rightarrow \R^d\times \R^d$ is {\sl vertically foliated} if
\begin{itemize}
\item its image is a product  $I^d\times J^d$ where $I$ and $J$ are two intervals of $\R$;
\item $\forall x\in {\cal U}, F(T^*_{\pi(x)}M\cap {\cal U})=F_1(x)\times J^d$.
\end{itemize}
\end{defi}

}

\begin{defi}\label{def-twist}
		An isotopy $(\phi_t)$ of conformally symplectic diffeomorphisms of    $\cal M$ twists the vertical if at every point $(t_0, x_0)\in \R\times \cal M$, there exists 
\begin{itemize}
\item $\varepsilon>0$;
\item a vertically foliated chart $F=(F_1, F_2): {\cal U}\rightarrow\R^{2d}$   such that $x_0\in\cal U$,  $F(\cal U)=(-a, a)^{d}\times (-a, a)^d$ and $F(x_0)=0$
\end{itemize}
that satisfy for all $t\in (t_0-\varepsilon, t_0+\varepsilon)$
\begin{itemize}
\item 
${\cal G}_t:=(\phi_t\circ\phi_{t_0}^{-1})\big(F^{-1}(\{ 0_{\R^d}\}\times (-\frac{a}{2}, \frac{a}{2})^d)\big) \subset{\cal U}$;
\item $F({\cal G_t})$ is the graph of a function $p\mapsto q=dg_t(p)$ where 
\begin{enumerate}
\item for $t\in [t_0, t_0+\varepsilon)$, $g_t$ is a convex function;
\item for $t\in (t_0-\varepsilon, t_0]$, $g_t$ is a concave function\footnote{We don't assume the strict concavity or convexity.}.
\end{enumerate}

\end{itemize}
\end{defi}
\begin{exa}  Assume that $a:I\rightarrow \R$ and $H:I\times \cal M\rightarrow \R$ are smooth functions and let us use the notation $H_t(x)=H(t, x)$. We assume  that the Hessian of $H$ restricted to every vertical fiber is positive definite\footnote{Observe that such a fiber is a linear space, hence the Hessian has an intrinsic meaning at every point.}. We define the time-dependent vector field $X_t$ of $\cal M$ by
$$i_{X_t}\omega=dH_t-a(t)\lambda.$$
Then the isotopy defined by $X_t$ is conformally symplectic and twists the vertical, see Proposition \ref{PHconvextwist}.
A subclass of examples is the class of discounted Tonelli flows, see e.g. \cite{MaroSor2017}.

~\newline
\begin{remk}
    In Proposition \ref{PtwistMAslov} we will prove that, when the isotopy $(\phi_t)$ twists the vertical, all the dynamical Maslov indices are non positive. 
\end{remk}

\end{exa}
\begin{theorem}\label{existencepoints0DMIbis}
	Let $\cl\subset {\cal M} $ be a Lagrangian submanifold that is {H-isotopic to a graph}. Let $(\phi_t)$ be an isotopy  of conformally symplectic diffeomorphisms of ${\cal M}$ that twists the vertical.

	Then there exists a constant $C\in \N^*$ and a point $x\in \cl$ such that  
	
	$$\forall t\in [0, +\infty), {\rm DMI}(T_x\cl, (\phi_s)_{s\in [0, t]})\in [-C, C]\,.$$
	In particular
	\[
	\mathrm{DMI}_\infty(x,(\phi_t)
	)=0\, .
	\]
	\end{theorem}
	Moreover, we deduce the following.
	\begin{theorem}\label{existencepoints0DMItris}
		Let $\cl\subset {\cal M} $ be a Lagrangian submanifold that is {H-isotopic to a graph}. Let $(\phi_t)$ be an isotopy  of conformally symplectic diffeomorphisms of ${\cal M}$ that twists the vertical and such that $\phi_{1+t}=\phi_t\circ\phi_1$. Let $x\in\cl$ be the point given by Theorem \ref{existencepoints0DMIbis}. Assume that the positive orbit of $x$ is relatively compact. Then there exists an ergodic $\phi_1$-invariant probability measure $\mu$ with compact support such that 
		$$\mathrm{DMI}(\mu, (\phi_t)
		)=0\,.$$
		Moreover, the support of $\mu$ is contained in the $\omega$-limit set of $x$.
\end{theorem}

{ Corollary \ref{cor on minimizing measure} explains why this statement is reminiscent of Ma\~n\'e and Mather theory for invariant measures of Tonelli Hamiltonians flows.
\begin{defi}
	A measure $\mu$ is minimizing for a Tonelli Hamiltonian flow if its dual measure $\nu$ on $TM$ is such that\[
	\int_{TM}L\,d\nu = \inf_{\rho}\int_{TM}L\,d\rho\, ,
	\]
	where $L$ is the associated Lagrangian function and the infimum is taken over all measures on $TM$ invariant by the Euler-Lagrange flow.
\end{defi}
\begin{cor}\label{cor on minimizing measure}
	Let $\cl\subset\cal M$ be a Lagrangian graph. Let $(\phi_t)$ be a Tonelli Hamiltonian flow. The invariant measure $\mu$ with compact support of zero asymptotic Maslov index given by Theorem \ref{existencepoints0DMItris} applied at $\cl$ is a Mather minimizing measure.
\end{cor}}

\begin{ques}
In the symplectic Tonelli case, when $\cl$ is not a graph, especially when its graph selector is not semi-concave, does Theorem \ref{existencepoints0DMItris} always give a minimizing measure?
\end{ques}

\begin{ques}
Without the Tonelli hypothesis, can we 
characterize the invariant measure of zero asymptotic Maslov index given by Theorem \ref{existencepoints0DMItris}?
\end{ques}

\begin{exa}
 At the beginning of this introduction, we dealt with the completely integrable case, where $\cal M$  is foliated by invariant graphs and where there are minimizing invariant measures in each of these graphs. But there are dissipative examples where there is only one measure with zero asymptotic Maslov index. In the case of the damped pendulum, see e.g. \cite{MaroSor2017}, there are only two invariant measures, one supported at  a sink with non-zero asymptotic index and one measure supported at a saddle hyperbolic fixed point, which has zero asymptotic Maslov index. Moreover, the only points that have zero asymptotic Maslov  index are the points that belong to the stable manifold of this saddle point. In this case, the Hausdorff dimension of the set of points with vanishing asymptotic Maslov index is 1. The next statement explain why it cannot be less in this setting.
\end{exa}

\begin{cor}\label{hausdorff dimension} Let $(\phi_t)$ be an isotopy  of conformally symplectic diffeomorphisms of ${\cal M}$ that twists the vertical. Assume that there exists $n$ closed 1-forms $\eta_1, \dots, \eta_n$ of $M$ such that no non-trivial linear combination of them vanishes, i.e.
	$$\forall (\lambda_1, \dots, \lambda_n)\in \R^n\backslash\{ 0_{\R^n}\}, \forall q\in M, \sum_{k=1}^n\lambda_k\eta_k(q)\not=0\,.$$
	Then 
	\[
	\text{dim}_H\Big( \{ x\in {\cal M} ;\ \text{\rm DMI}(x,(\phi_t)
	)=0 \} \Big)\geq n,
	\]
	where $\text{dim}_H(U)$ denotes the Hausdorff dimension of a set $U$.
\end{cor}
\begin{remk}When $M$ is the $d$ dimensional torus, this statement allows to bound from below by $d$ the Hausdoff dimension of the set of points with zero asymptotic Maslov index.
\end{remk}

We now give a by-product of the proof of Theorem \ref{Tppal}. This proof  relies on spectral  invariants that come from the symplectic topology, in particular graph selectors that were introduced by {\color{black}Chaperon and Sikorav, see \cite{Chap1991}, \cite{OttVit94} or \cite{PPS2003}}. We will see in the proof that the closed 1-form $\eta$ and the Lipschitz function $u$ in Theorem \ref{Tppal} only depend on $\phi_1(\cl)$ and not on the isotopy and will deduce, after introducing in Section \ref{angular MI} the angular Maslov index,  the following statement, which expresses the independence of the dynamical Maslov index from the isotopy.

\begin{proposition}\label{Indisotopybis}
Let $(\phi_{1, t})$ and $(\phi_{2, t})$ be two isotopies of conformally symplectic diffeomorphisms of ${\cal M}$ such that $\phi_{1, 0}=\phi_{2, 0}={\rm Id}_{\cal M}$ and $\phi_{1, 1}=\phi_{2, 1}$.  Then for every Lagrangian subspace $L$ of $T{\cal M}$ such that $L$ and $D\phi_{1, 1}(L)$ are transverse to the vertical foliation, we have
$${\rm DMI}(L, (\phi_{1, t})_{t\in[0, 1]})= {\rm DMI}(L, (\phi_{2, t})_{t\in[0, 1]}).
$$
 	
\end{proposition}
\begin{remk} For ease of reading, we have chosen not to deal with angular Maslov index in this introduction.
The statement given in Section \ref{angular MI} is more precise, because it deals with the angular Maslov index for every Lagrangian subspace of $T{\cal M}$.
\end{remk}
\medskip

\noindent {\it Organisation of the paper.} Section \ref{section def MI} is devoted to the definition of the Maslov index and the dynamical Maslov index. We show that the twist hypothesis forces the index to be non positive. The invariance under symplectic reduction of the Maslov index is discussed following \cite{Vit1987}. In Section \ref{section gfqi} we prove that any Lagrangian path contained in a Lagrangian submanifold and whose endpoints project on the graph selector has zero Maslov index. This result is fundamental to prove Theorem \ref{Tppal}, whose proof occupies Section \ref{section heart}. The angular Maslov index is introduced in Section \ref{angular MI}, where also its relation with the Maslov index is detailed. Finally, Section \ref{section applications} is devoted to the proofs of the main outcomes presented in the introduction.

\medskip

\noindent {\it Acknowledgements.} The authors are grateful to Patrice Le Calvez for pointing out the link with Schwartzman's article \cite{Schwartzman1957}.
\section{On Maslov index}\label{section def MI}
\subsection{Some reminders on Maslov index}\label{ssremindersMslov}
Let $\cal M$ be a $2d$-dimensional  symplectic manifold that admits a Lagrangian foliation $\cal V$. We denote by  $V(x)=V_x:=T_x\cal V$ its associated Lagrangian bundle. Let $p:T{\cal M}\rightarrow \cal M$ be the  canonical projection. 
Let $\Lambda(\cal M)$ be the Grassmanian of Lagrangian subspaces of $T\cal M$.  We recall that $\Lambda(\cal M)$ is a smooth manifold with dimension $2d+\frac{d(d+1)}{2}$. The fibered singular cycle associated to $\cv$ is the set 
$$\Sigma(\cal M)=\{ L\in \Lambda(\cal M) :\  L\cap V_{p(L)}\not=\{ 0\}\}.$$
Every fiber $\Sigma_x(\cal M)$ of $\Sigma(\cal M)$ is a  cooriented algebraic singular hypersurface of $\Lambda_x(\cal M)$, see e.g. \cite{McDuffSalamon2017}, \cite{RobbinSalamon1993}. Hence $\Sigma(\cal M)$ is a cooriented singular hypersurface of $\Lambda(\cal M)$.\\
{ The singular locus of $\Sigma(\cal M)$ is then $\{ L\in\Lambda(\cal M) :\ \mathrm{dim}(L\cap V_{p(L)})\geq 2 \}$ and the regular locus is \begin{equation}\label{Esinglocus}
\Sigma_1:=\{ L\in\Lambda :\ \mathrm{dim}(L\cap V_{p(L)})=1 \}.
\end{equation} }
Once a coorientation of $\Sigma(\cal M)$ is fixed, it is classical to  associate to every continuous loop
$\Gamma:\T\rightarrow \Lambda (\cal M)$ its {\sl Maslov index} ${\rm MI}(\Gamma)$,  that satisfies the following properties:
\begin{itemize}
\item two homotopic loops have the same Maslov index;
\item if $\Gamma$ is a loop that avoids the singular locus   and is { topologically} transverse to the regular one 
of $\Sigma(\cal M)$, then ${\rm MI}(\Gamma)$ is the number of signed intersections of $\Gamma$ with $\Sigma(\cal M)$ with respect to the chosen coorientation;
\item every loop is homotopic to a smooth loop that avoid the singular locus and is transverse to the regular locus.
\end{itemize}
 An arc in $\Lambda(\cal M)$ is an immersion $\Gamma:[0,1]\to\Lambda(\cal M)$. In particular, $\Gamma([0,1])$ does not have self-intersections. By smooth arc, we mean a $C^\infty$ arc.
When $\Gamma: [0, 1]\rightarrow \Lambda(\cal M)$ is an arc 
	whose endpoints are in  $\Lambda({\cal M})\backslash \Sigma({\cal M})$, following 
	{\color{black} Duistermaat \cite[Page 183]{Dui76}},  
	we can concatenate $\Gamma$ with an arc $\Gamma_1$ that connects $\Gamma (1)$ to $\Gamma(0)$ in $\Lambda({\cal M})\backslash \Sigma({\cal M})$. The Maslov index of $\Gamma$ is the Maslov index of this loop, which is independent from the choice of $\Gamma_1$ {
	since $\Gamma_1$ is in $\Lambda(\M)\setminus\Sigma(\M)$}. 

~\newline
	\begin{remk}\label{rmq path transverse}
		If $\Gamma:[0,1]\to\Lambda(\cm)$ is an arc contained in $\Lambda(\cm)\setminus\Sigma(\cm)$, i.e. $\Gamma(t)\cap\Sigma_{p\circ\Gamma(t)}=\{0\}$ for every $t$, then its Maslov index ${\rm MI}(\Gamma)$ is zero.
	\end{remk}

\subsection{Coorientation of $\Sigma^1$}\label{subsec defi MI} 
We now give some details concerning the singular and regular loci of $\Sigma (\cal M)$ and explain our choice of coorientation of $\Sigma(\cal M)$. {\color{black}For more details, see for example \cite{Dui76}.} For ease of reading, we denote $\Sigma (\cal M)$ (resp. $\Lambda(\cal M)$) by $\Sigma$ (resp. $\Lambda$).

Then  $\Sigma$ is an algebraic subvariety of $\Lambda$ that  is the union of   
\begin{itemize}
\item the regular locus that is the smooth submanifold of codimension 1 and is defined in \eqref{Esinglocus}, 
\item the  boundary of $\Sigma_1$, i.e.  the singular locus  $\Sigma\setminus\Sigma^1$, that is a finite union of submanifolds with   codimension at least 3.
\end{itemize}
Since every loop is homotopic to a smooth loop avoiding the singular locus and intersecting transversally the regular one { and since two homotopic loops have the same Maslov index}, we just have to define the coorientation at points of $\Sigma^1$. To do that, we introduce the notion of height in a symplectic vector space $(E^{2d} , \Omega)$.\\
 We fix a reference Lagrangian subspace $V$ of $E$ and denote by $P^V$ the canonical projection on the quotient vector space $E/V$. If $L_1$, $L_2$ are two Lagrangian subspaces of $E$ that are transverse to $V$, we define the height of $L_1$ above $L_2$ with respect to $V$, see \cite{Arn08}, as follows.

  	\begin{defi}
  	Let $L_1,L_2\subset E$ be two Lagrangian subspaces both transverse to $V$. The {\sl height} of $L_2$ above $L_1$ with respect to $V$ is the quadratic form $$\mathcal{Q}_V(L_1,L_2):E/V\to\R$$ defined by 
  $$
  	\forall v\in E/V, \mathcal{Q}_V(L_1,L_2)(v):=\Omega(({P}^V|_{L_1})^{-1}(v),({P}^V|_{L_2})^{-1}(v)).
  	$$
  	\end{defi}
With the hypotheses of  this definition, the kernel of $ \mathcal{Q}_V(L_1,L_2)$ is isomorphic to $L_1\cap L_2$. In particular, $L_1$ is transverse to $L_2$ if and only if $Q_V(L_1,L_2)$ is non degenerate.\\
We have 
\begin{itemize}
\item If  $L_1,L_2,L_3$ are Lagrangian subspaces in $E$, all transverse to $V$, it holds, see \cite{Arn08}, 
\begin{equation}\label{prop:quadr:form}
\mathcal{Q}_V(L_1,L_3)=\mathcal{Q}_V(L_1,L_2)+\mathcal{Q}_V(L_2,L_3).
\end{equation}
\item if $V$, $K$, $L$ are Lagrangian subspace of $E$ such that each of them is transverse to the two others, then { $\mathcal{Q}_V(K,L)\circ P^V|_{L}=-\mathcal{Q}_K(V,L)\circ P^K|_{L}$ and then}  $\mathcal{Q}_V(K,L)$ and $-\mathcal{Q}_K(V,L)$ have the same signature.\\
{ Let us prove  that $\mathcal{Q}_V(K,L)\circ P^V|_{L}=-\mathcal{Q}_K(V,L)\circ P^K|_{L}$. For $\ell\in L$, there exists a unique pair of vectors $v\in V, k\in K$ such that $\ell=v+k$. then we have
\begin{itemize}[label=$\bullet$]
    \item $\mathcal{Q}_V(K,L)\circ P^V(\ell)=\Omega(k, \ell)=\Omega(k, v)$;
    \item $-\mathcal{Q}_K(V,L)\circ P^K(\ell)=-\Omega(v, \ell)=-\Omega(v, k)=\Omega(k, v)
    $.
\end{itemize}
}
\item if $L$ and $K$ are Lagrangian subspaces that are transverse to $V$ and if $\phi:E\righttoleftarrow$ is a symplectic isomorphism, then ${\cal Q}_V(K, L)$ has same signature as ${\cal Q}_{\phi(V)}(\phi(K), \phi(L))$.
\end{itemize}
{ We now describe the local coorientation of $\Sigma_1$ that we will use.} Let us fix $L_0\in\Sigma_1$ and let $x_0:=p(L_0)$. We have $\dim(L_0\cap V_{x_0})=1$. We fix a Darboux  chart $F=(F_1, F_2): {\cal U}\rightarrow\R^{2d}$ at $x_0$  such that $\cal U$ is a small neighborhood of $x_0$ in $\M$, $F(\cal U)=[a, b]^{d}\times [a, b]^d$ and $DF_{2|L_0}$ is injective  and $$\forall x\in {\cal U}, F({ \cV(x)}\cap \cal U)= F_1(x)\times [a, b]^d.
$$
{ Let us explain why such a chart exists. Using Theorem 7.1 of \cite{Weinstein1971}, we can map locally the foliation $\cV$ onto the vertical foliation of $\R^{2d}$ by a symplectic chart $(U, \Phi)$. Then, composing with a symplectic isomorphism $\psi_t(x, y)=(x, y+tx)$ of $\R^d\times\R^d$, for some $t\in\R$, we obtain a new chart $F=(F_1, F_2)$ that maps $\cV$ onto the vertical foliation such that $DF(L_0)$ is transverse to $\{0\}\times \R^d$ and then $DF_{2|L_0}$ is injective.

}

We denote by ${\cal K}$ the  Lagrangian foliation with leaves $F^{-1}([a, b]^d\times \{y_0\}$). Then it is  transverse to the vertical bundle $V$. Moreover, $T_{x_0}{\cal K}$ and $L_0$ are transverse, since ${DF_2}_{\vert L_0}$ is injective. We denote by $K$ the tangent bundle to ${\cal K}$. { Because $\dim(L_0\cap V)=1$,}  the kernel of ${\cal Q}_{K}(V, L_0)$ is 1-dimensional. We denote by $n$ the index\footnote{The \emph{index} of a quadratic form is the maximum dimension of a subspace of $E$ on which the quadratic form is negative definite.} of ${\cal Q}_{K}(V, L_0)$. We define $${\cal P}_1=\{ L\in \Lambda\backslash \Sigma;\  p(L)\in {\cal U},\ L\pitchfork K,\ {\rm index\,}{\cal Q}_{K}(V, L)=n\}$$
 and $$ {\cal P}_2=\{ L\in \Lambda\backslash \Sigma;\    p(L)\in {\cal U},\ L\pitchfork K,\  {\rm index\,}{\cal Q}_{K}(V, L)=n+1\}.$$
Observe that ${\cal P}_1$ and ${\cal P}_2$ are connected and that ${\cal P}_1\cup{\cal P_2}\cup \Sigma_1$ is a neighbourhood of $L_0$ in $\Lambda$. Hence ${\cal P}_1$ and ${\cal P}_2$ define locally a coorientation\footnote{Let $\gamma\subset \cal P_1\cup\cal P_1\cup \Sigma_1$ be a path from $\cal P_2$ to $\cal P_1$, crossing $\Sigma_1$ transversally once at $\gamma(t)$. Then $\gamma'(t)\in \R_+ N$, where $N$ is a normal vector field to $\Sigma_1$ and determines a coorientation of $\Sigma_1$ at $L_0$.} 
of $\Sigma_1$ at $L_0$. To be sure that we obtain a global coorientation of $\Sigma$, we have to prove that this local coorientation is independent from the choice of our foliation ${\cal K}$. We just have to look at what happens in the fiber $\Lambda_{x_0}$ for different choices of $K_{x_0}$. In other words, we will prove a result in a fixed symplectic vector space $(E, \Omega)$.
\begin{proposition}\label{Pcoorientation}
Let $V$, $L_0$ be two Lagrangian subspaces of $E$ such that $\dim (L_0\cap V)=1$. Let $K_1$, $K_2$ be two Lagrangian subspaces of $E$ that are transverse to $L_0$ and $V$. We denote by $n_i$ the index of ${\cal Q}_{K_i}(V, L_0)$. There exists a neighbourhood $\cal U$ of $L_0$ in the Lagrangian Grassmannian of E such that \\
$\{ L\in {\cal U};\ L\pitchfork K_1,\  {\rm index\,}{\cal Q}_{K_1}(V, L)=n_1+1\}=$\hglue 5truecm\\
\hglue 5truecm$\{ L\in {\cal U};\ L\pitchfork K_2,\  {\rm index\,}{\cal Q}_{K_2}(V, L)=n_2+1\}
$\\
and
$$\{ L\in {\cal U};\ L\pitchfork K_1,\  {\rm index\,}{\cal Q}_{K_1}(V, L)=n_1\}=\{ L\in {\cal U};\ L\pitchfork K_2,\  {\rm index\,}{\cal Q}_{K_2}(V, L)=n_2\}.$$
\end{proposition}
\begin{proof} Because, for $j=1,2$, $V$ and $K_j$ are transverse Lagrangian subspaces, every basis $(e_1, \dots, e_d)$ of $V$ can be completed in a symplectic basis  $(e, f^j)=(e_1, \dots, e_d, f^j_1, \dots, f^j_d)$ of $E$ such that   $f^j_i\in K_j$. Then, every Lagrangian subspace $L$ of $E$ that is close enough to $L_0$ is the graph in this basis of a $d\times d$  symmetric matrix $S_j^L$ that continuously depends on $L$ and is close to $S_j^{L_0}$. { We identify $V$ with $\R^d$ via the basis $(e_i)$.} \\
	As $\dim\ker S_j^{L_0}=1$, we have $\R^d=\R \ell_j(L_0)+E_j(L_0)$ where $\ker S_j^{L_0}=\R\ell_j(L_0)$ and $E_j(L_0)=(\R \ell_j(L_0))^\bot$ is the orthogonal of $\ker S_j^{L_0}$ for the usual euclidean scalar product, i.e. the sum of the eigenspaces for the non-zero eigenvalues. Observe that we can choose $\ell_1(L_0)=\ell_2(L_0)$. For $L$ in some neighbourhood ${\cal U}$ of  $L_0$, $S_j^L$ has a spectral gap with one eigenvalue $\lambda(S_j^L)$ close to $0$ and the others far away from $0$. Hence we can continuously extend $\ell_j(L)$ and $E_j(L)$ for $L$ close to $L_0$ in such a way that $\ell_j(L)$ is an eigenvector for the eigenvalue that is close to $0$, and $E_j(L)$ is $(\R\ell_j(L))^\perp$. Moreover, the signature of the restriction of $S_j^L$ to $E_j(L)$ remains equal to its value for $L=L_0$ if $\cal U$ is small enough. \\
	The matrix of ${\cal Q}_{K_j}(V, L)$ in the basis $(P^{K_j}(e_1), \dots, P^{K_j}(e_d))$ of $E/K_j$ is $S_j^L$ and then to estimate the index of ${\cal Q}_{K_j}(V, L)$, we only need to know the sign of $\lambda(S_j^L)$.\\
	We recall that when  $L\in {\cal U}$ is transverse to $V$,   { we have ${\cal Q}_{K_j}(V, L)\circ \big(P^{K_j}|_{L}\big)^{-1}=-{\cal Q}_{V}(K_j, L)\circ\big( P^V|_{L}\big)^{-1}$}. The matrix of $-{\cal Q}_{V}(K_j, L)$ in the basis $(P^V(f_1^j),\dots   , P^V(f_d^j))$ is  $\Big(S_j^L\Big)^{-1}$ and thus we are reduced to estimate the sign of the  eigenvalue of $\Big(S_j^L\Big)^{-1}$ that has the largest absolute value. Observe that $P^V(f_i^1)=P^V(f_i^2)$.
	We denote by $S$ the matrix of $ {\cal Q}_{V}(K_1,K_2)$ in the same basis and we deduce from \eqref{prop:quadr:form} that 
	\begin{equation}\label{Ematsym}-\Big(S_1^L\Big)^{-1}=-\Big(S_2^L\Big)^{-1}+S
	\end{equation}
	Let us denote by $\|\cdot\|_2$ the usual Euclidean norm  on $\R^d$ and let us endow the set of $d$-dimensional matrices with the associated norm defined by
	$$\| S\|=\sup_{\|v\|_2=1} \|Sv\|_2\,.$$ Then if $\,{\cal U}$ is small enough, there exists $C>\| S\|$ such that for every $L\in {\cal U}$,
	$\big(\lambda(S_j^L)\big)^{-1}$ is the only eigenvalue of $\Big(S_j^L\Big)^{-1}$ whose absolute value is larger than $3C$ and $C$ is an upper bound of the modulus of all the other eigenvalues of $\Big(S_j^L\Big)^{-1}$.
	
	Let us prove  that  $\lambda(S_1^L)$ and $\lambda(S_2^L)$ have the same sign. Let $v\in\R^d$ be an eigenvector of $S_1^L$ for the eigenvalue $\lambda(S_1^L)$. Then there exists $v_1, v_2\in\R^d$ that are mutually orthogonal such that {\color{black}$v=v_1+v_2$,} $S_2^Lv_1=\lambda(S_2^L)v_1$ and $v_2$ is orthogonal to the eigenspace of $S_2^L$ for $\lambda(S_2^L)$.  Using \eqref{Ematsym}, we obtain
	$$\big(\lambda(S_1^L)\big)^{-1}-\big(\lambda(S_2^L)\big)^{-1}\frac{\| v_1\|_2^2}{\| v\|_2^2}=\frac{v_2^T}{\| v\|_2} \big(S_2^L\big)^{-1}\frac{v_2}{\| v\|_2}- \frac{v^T}{\| v\|_2}S\frac{v}{\| v\|_2}
	$$
	Observe that the  absolute value of the right-hand term is less than $2C$. If $\lambda(S_1^L)$ and $\lambda(S_2^L)$ have different signs, then the absolute value of the left-hand term is larger than the absolute value of  $\big(\lambda(S_1^L)\big)^{-1}$, then larger than $3C$, which provides a contradiction.
\end{proof}
In order to define the Maslov index, we first introduce the notions of positive (resp. negative) arc. Recall that $K$ denotes the tangent bundle to $\cal K$, where $\cal K$ is the Lagrangian foliation with leaves $F^{-1}([a,b]^d\times\{y_0\})$.
\begin{defi} With the same notation, an arc $\Gamma:(-\varepsilon_0, \varepsilon_0)\rightarrow \Lambda$ such that $$\Gamma((-\varepsilon_0,\varepsilon_0))\cap \Sigma=\Gamma((-\varepsilon_0,\varepsilon_0))\cap \Sigma_1=\{ \Gamma(0)\} =\{L_0\} $$  and that is topologically transverse to $\Sigma^1$ is {\sl positive} if there exists $\varepsilon>0$ such that
\begin{itemize}
\item for every $t\in (-\varepsilon, 0)$, ${\rm index}({\cal Q}_{K}(V, \Gamma(t)))={\rm index}({\cal Q}_{K}(V, L_0))+1$;
\item for every $t\in (0, \varepsilon)$, ${\rm index}({\cal Q}_{K}(V, \Gamma(t)))={\rm index}({\cal Q}_{K}(V, L_0))$.
\end{itemize}
Respectively, an arc $\Gamma:(-\varepsilon_0,\varepsilon_0)\to\Lambda$ is {\sl negative} if $\Gamma\circ(-{\rm Id})$ is positive.
\end{defi}

\begin{remk} This is equivalent to 
\begin{itemize}
\item for every $t\in (-\varepsilon, 0)$, ${\rm index}({\cal Q}_{V}(K, \Gamma(t)))=d-{\rm index}({\cal Q}_{K}(V, L_0))-1$;
\item for every $t\in (0, \varepsilon)$, ${\rm index}({\cal Q}_{V}(K, \Gamma(t)))=d-{\rm index}({\cal Q}_{K}(V, L_0))$.
\end{itemize}
\end{remk} 
\begin{defi} Let $\Gamma:[a, b]\rightarrow \Lambda$ be an arc. 
\begin{itemize}
\item A $t\in[a,b]$ is a {\sl crossing} for $\Gamma$ if $\Gamma(t)\in\Sigma$.
\item The arc $ \Gamma$ is in {\sl general position} with respect to $\Sigma$ if $\Gamma(a),\Gamma(b)\in\Lambda\setminus\Sigma$ and the path $\Gamma$ is { topologically} transverse to  $\Sigma$.
{\item The arc $ \Gamma$ is in {\sl D-general position} with respect to $\Sigma$ if $\Gamma(a),\Gamma(b)\in\Lambda\setminus\Sigma$ and the path $\Gamma$ is  transverse (in the differentiable sense) to  $\Sigma$.}

\end{itemize} 	
\end{defi}
\begin{remk}
    If $\Gamma:[a,b]\to\Lambda$ is in general position with respect to $\Sigma$, then each crossing for $\Gamma$ is isolated.
    {Let $[a, b]$ be fixed and let $k\in\N^*\cup\{ \infty\}$. Then, the set of $C^k$ arcs $\Gamma:[a, b]\to \Lambda$ that are in D-general position with respect to $\Sigma$ is open for the $C^1$-topology.}
\end{remk}

Let $\Gamma:[a,b]\to\Lambda$ be an arc in general position with respect to $\Sigma$. A crossing $t$ is called \textit{positive}, respectively \textit{negative}, if there exists $\epsilon>0$ such that the arc $\Gamma\vert_{[ t-\epsilon, t+\epsilon]}:[ t-\epsilon, t+\epsilon]\to\Lambda$ is positive, respectively negative.


\begin{defi}\label{defi MI for arc}
Let $\Gamma:[a,b]\to\Lambda$ be an arc in general position with respect to $\Sigma$. The Maslov index of $\Gamma$ with respect to {$V$ or $\cV$} is
\[
\mathrm{MI}(\Gamma) :=\]\[\mathrm{Card}\{t :\, t\text{ is a positive crossing for }\Gamma\}-\mathrm{Card}\{t:\, t \text{ is a negative crossing for }\Gamma\}\, .
\]
\end{defi}
The notion of Maslov index can be extended to Lagrangian paths that are not in general position.
\begin{defi}
Let $\Gamma:[a,b]\to\Lambda$ be a path such that $\Gamma(a),\Gamma(b)\in\Lambda\setminus\Sigma$ (not necessarily in general position with respect to $\Sigma$). Let { $\tilde\Gamma:[a,b]\to \Lambda$ be a smooth arc that is $C^1$-close to $\Gamma$ 
and is in general position with respect to $\Sigma$. }  
Then
\[
\mathrm{MI}(\Gamma):=\mathrm{MI}(\tilde\Gamma)\, .
\]
\end{defi}
For the existence of the perturbation $\tilde\Gamma$ of $\Gamma$ and for the independence of the previous definition from the choice of $\tilde\Gamma$ we refer to \cite{Arn72} or \cite{Capleem}.
~\newline

\begin{remk}\label{rmk invariance diffeo cs}
	Let $\phi$ be a conformally symplectic diffeomorphism on $\M$. Let $\Gamma:[a.b]\to\Lambda$ be a smooth path such that $\Gamma(a),\Gamma(b)\notin\Sigma$. Then $$D\phi(\Gamma):[a,b]\ni t\to D\phi(\Gamma(t))\in\Lambda$$ is still a smooth path such that $D\phi(\Gamma)(a),D\phi(\Gamma)(b)$ do not belong to
	\[
	\{ L\in\Lambda :\ L\cap D\phi(V)_{p(L)}\neq\{0\} \}\, ,
	\]
	where $D\phi(V)_x$ is the tangent bundle associated to the Lagrangian foliation $\phi(\cv)$.
	Then 
	the Maslov index $\mathrm{MI}(\Gamma)$, calculated with respect to the Lagrangian foliation $\cv$, is equal to the Maslov index $\mathrm{MI}(D\phi(\Gamma))$, calculated with respect to the Lagrangian foliation $\phi(\cv)$.
	
	Moreover, if $\phi(\cv)=\cv$, then \[\mathrm{MI}(\Gamma)=\mathrm{MI}(D\phi(\Gamma))\, .\] 
	In particular, for $\M=T^*M$, the Maslov index is invariant by \textit{vertical translations}, that is by any diffeomorphism of the form $\phi(p)=p+\eta\circ\pi(p)$, where $\eta$ is a closed 1-form in $T^*M$.
\end{remk}
\subsection{Dynamical Maslov index}
We now give the definition of dynamical Maslov index.

\begin{defi}
	Let $(\M,\omega)$ be a symplectic manifold that admits a Lagrangian foliation $\cv$. Let $(\phi_t)
	$ be an isotopy of conformally symplectic diffeomorphisms of $\M$. Let $L\in\Lambda$ and $[\alpha,\beta]\subset \R$ be such that $D\phi_\alpha(L),D\phi_\beta(L)\notin\Sigma$. Then
	\[
	\mathrm{DMI}(L,(\phi_t)_{t\in[\alpha,\beta]}):= \mathrm{MI}(\Gamma)\, ,
	\]
	where $\Gamma$ is the Lagrangian path $[\alpha,\beta]\ni t\mapsto \Gamma(t):=D\phi_t(L)\in\Lambda$ and the Maslov index $\mathrm{MI}(\Gamma)$ is calculated with respect to the Lagrangian foliation $\cv$.
\end{defi}

\subsection{Twist and Maslov index}

In this section, we work in $T^*M$ and we denote ${\cal V}(x)=T_x^*M$.

{ In the introduction, we gave the definition of an isotopy which twists the vertical. We can enhance this in the following way (we adopt the same notations $F$, $g_t$ and $\cal G_t$ as in the definition of twist of the  vertical).
{

\begin{defi}
An isotopy $(\phi_t)$ of conformally symplectic diffeomorphisms of $\M$ \emph{strictly} twists the vertical if it twists the vertical and 
at every $t_0\in\R$ \begin{itemize}
    \item for all $t\in(t_0,t_0+\epsilon)$ the image $F(\cal G_t)$ is the graph of a function $p\mapsto q=dg_t(p)$ where $g_t$ is a strictly convex function { i.e. such that $d^2g_t$ is positive definite};
    \item for all $t\in(t_0-\epsilon,t_0)$ the image $F(\cal G_t)$ is the graph of a function $p\mapsto q=dg_t(p)$ where $g_t$ is a strictly concave function { i.e. such that $d^2g_t$ is negative definite}.
    \end{itemize}
\end{defi}
}

Observe that the condition of convexity depends on the charts we choose (even if the property of twisting the vertical is invariant by symplectic conjugation that preserves the vertical foliation). This is { a motivation to}    give a result of the twist property that doesn't use any chart.

\begin{proposition}\label{Ptwistintrinsec}
Let $(\phi_t)$ be an isotopy of conformally symplectic diffeomorphisms of $T^*M$ that twists the vertical. Let $x\in T^*M$ and let $t_0\in\R$. We denote $x_t=\phi_t(x)$.
Let $K$ be a continuous Lagrangian bundle that is defined in a neighbourhood of $x_{t_0}$ and is transverse to the vertical bundle.   Then there exists $\varepsilon>0$ such that
\begin{itemize}
    \item $\forall t\in (t_0, t_0+\varepsilon),  Q_{K(x_t)}(D\phi_t\circ \big(D\phi_{t_0}\big)^{-1}V(x_{t_0}), V(x_t))
    $  is a negative semi-definite quadratic form;
    \item $\forall t\in (t_0-\varepsilon, t_0),  Q_{K(x_t)}(D\phi_t\circ \big(D\phi_{t_0}\big)^{-1}V(x_{t_0}), V(x_t))
    $   is a positive semi-definite quadratic form.
    
\end{itemize}
Moreover, when $(\phi_t)$ strictly twists the vertical,    the considered quadratic forms are negative definite or positive definite.
\end{proposition}
\begin{proof}[Proof of Proposition \ref{Ptwistintrinsec}]
    We fix $\varepsilon>0$ and
a vertically foliated chart $F=(F_1, F_2): {\cal U}\rightarrow\R^{2d}$    such that  $x_{t_0}\in\cal U$, $F(x_{t_0})=0$   and for $t\in (t_0-\varepsilon, t_0+\varepsilon)$
\begin{itemize}
\item 
${\cal G}_t:=(\phi_t\circ\phi_{t_0}^{-1})\big(F^{-1}(\{ 0_{\R^d}\}\times (-\frac{a}{2}, \frac{a}{2})^d)\big) \subset{\cal U}$;
\item $F({\cal G_t})$ is the graph of a function $p\mapsto q=dg_t(p)$ where 
\begin{enumerate}
\item for $t\in [t_0, t_0+\varepsilon)$, $g_t$ is a convex function;
\item for $t\in (t_0-\varepsilon, t_0]$, $g_t$ is a concave function.
\end{enumerate}
\end{itemize}
As previously, we denote by ${\cal K}$ the  Lagrangian foliation with leaves $F^{-1}([-a, a]^d\times \{y_0\}$) and  by $K$ its tangent bundle.  For $t\in(t_0,t_0+\epsilon)$ (resp. $(t_0-\epsilon, t_0)$), the quadratic form  $${\cal Q}_{K(x_t)}(D(\phi_t\circ \phi_{t_0}^{-1})V(x_{t_0}), V(x_t))=-{\cal Q}_{K(x_t)}( V(x_t), D(\phi_t\circ \phi_{t_0}^{-1})V(x_{t_0}))$$ expressed in the chart $F$ is just $-D^2g_t(F_2(x_{t_0}))$ that is   a negative (resp. positive) semi-definite quadratic form  because the isotopy twists  the vertical.\\
When the isotopy strictly twists the vertical, we obtain in this case a  negative (resp. positive) definite quadratic form.

Observe that the bundle $K$ that we use in the proof is not necessarily the same bundle as   in the statement. But because the two are transverse to the vertical foliation  and we consider the height between the vertical and Lagrangian subspaces that are close to the vertical ($\varepsilon$ is small), the two indices are the same (we can build an isotopy between the two bundle that won't change the signature).
\end{proof}

\begin{proposition}\label{Ptermchpvect}
Let $(X_t)$ be a conformally symplectic vector field that generates an isotopy  of conformally symplectic diffeomorphisms of $T^*M$. We assume that at every $x\in T^*M$, there exists a vertically foliated chart $F=(F_1, F_2):\cU\to \R^{2d}$ such that if we write the vector field $X=(X_q, X_p)$ in this chart, then $\partial_pX_q$, which is always symmetric because the vector field is symplectic,  is positive definite.

Let $x\in T^*M$ and  let us  denote $x_t=\phi_t(x)$. Let $t_0\in\R$.
Let $K$ be a continuous Lagrangian bundle that is defined in a neighbourhood of $x_{t_0}$ and is transverse to the vertical bundle.   Then there exists $\varepsilon>0$ such that
\begin{itemize}
    \item $\forall t\in (t_0, t_0+\varepsilon),  Q_{K(x_t)}(D\phi_t\circ \big(D\phi_{t_0}\big)^{-1}V(x_{t_0}), V(x_t))$ is positive definite;
    \item $\forall t\in (t_0-\varepsilon, t_0),  Q_{K(x_t)}(D\phi_t\circ \big(D\phi_{t_0}\big)^{-1}V(x_{t_0}), V(x_t))$ is negative definite.
    \end{itemize}
 \end{proposition}

 \begin{proof}[Proof of Proposition \ref{Ptermchpvect}] As noticed in the proof of Proposition \ref{Ptwistintrinsec}, we only need to prove the result for the tangent space $K$ to  Lagrangian foliation  ${\cal K}$  with leaves $F^{-1}([-a, a]^d\times \{y_0\}$). 
 
 In the chosen chart, the Jacobian matrix of $X$ is
 $$DX(x_{t_0})=\begin{pmatrix}\partial_qX_q(x_{t_0})&\partial_pX_q(x_{t_0})\\
 \partial_qX_p(x_{t_0})&\partial_pX_p(x_{t_0})
 \end{pmatrix}$$
 and if we denote $$D\phi_t(x) \big(D\phi_{t_0}\big)^{-1}(x_{t_0})=\begin{pmatrix}a_t&b_t\\ c_t&d_t
 \end{pmatrix}$$
 then we have
\begin{equation}\label{Ehamlin}
\begin{cases}
  \dot b_t=\partial_qX_qb_t+\partial_pX_qd_t\\
  \dot d_t=\partial_qX_pb_t+\partial_pX_pd_t\, .
 \end{cases}
     \end{equation}
{    Hence uniformly in $x$ it holds   $d_t= {\bf 1}_d+o(t-t_0)$ and    $b_t=(t-t_0)\partial_pX_q(x_{t_0})+o((t-t_0)^2)$, which gives  $b_t(d_t)^{-1}= (t-t_0)\partial_pX_q(x_{t_0})+o((t-t_0)^2)$. Because $b_t(d_t)^{-1}$ is the matrix of $$Q_{K(x_t)}(D\phi_t\circ \big(D\phi_{t_0}\big)^{-1}V(x_{t_0}), V(x_t))$$ in the chart, this gives the wanted result.}
 \end{proof}

}
\begin{proposition}\label{PtwistMAslov} Let $(\phi_t)$ be an isotopy of conformally symplectic diffeomorphisms of $T^*M$ that twists the vertical. Then if $L\in\Lambda$ and $[\alpha, \beta]\subset \R$ are such that $D\phi_\alpha(L), D\phi_\beta(L)\notin \Sigma$, then
$${\rm DMI}(L, (\phi_{ t})_{t\in[\alpha, \beta]})\leq 0.$$
\end{proposition}

\begin{proof}[Proof of Proposition  \ref{PtwistMAslov}] 
{Let us first assume that $(\phi_t)$ is an isotopy {that satisfies the conclusion of Proposition\ref{Ptwistintrinsec} (with definite quadratic forms) in a neighborhood of $(D\phi_tL)_{t\in[\alpha,\beta]}$.} Perturbing $L$, we can assume that $(D\phi_{ t}L)_{t\in[\alpha, \beta]}$ intersects $\Sigma$ eventually only at the regular locus $\Sigma_1$. {We will prove that 
this implies that   $(D\phi_{ t}L)_{t\in[\alpha, \beta]}$ is actually {topologically} transverse to $\Sigma$ and that the Maslov index is non-positive.}}

Let $t_0\in [\alpha, \beta]$ be such that  $D\phi_{t_0}L\in \Sigma_1$. { We introduce $x_0=p(L)$ and $x_t:=\phi_t(x_0)$.
Let $K$ be a continuous Lagrangian bundle that is defined in a neighbourhood of $x_{t_0}$ and is transverse to the vertical bundle.   Then by hypothesis there exists $\varepsilon>0$ such that
\begin{itemize}
    \item $\forall t\in (t_0, t_0+\varepsilon),  Q_{K(x_t)}(D\phi_t\circ \big(D\phi_{t_0}\big)^{-1}V(x_{t_0}), V(x_t))$ is positive definite;
    \item $\forall t\in (t_0-\varepsilon, t_0),  Q_{K(x_t)}(D\phi_t\circ \big(D\phi_{t_0}\big)^{-1}V(x_{t_0}), V(x_t))$ is negative definite.
    
\end{itemize}

}

Then, for $t\in [t_0, t_0+\varepsilon)$, if  $L_t=D\phi_tL$,  we have\\
$${\cal Q}_{K(x_t)}(L_t, V(x_t))=$$
$${\cal Q}_{K(x_t)}(D(\phi_t\circ \phi_{t_0}^{-1})L_{t_0}, D(\phi_t\circ \phi_{t_0}^{-1})V(x_{t_0}))+{\cal Q}_{K(x_t)}(D(\phi_t\circ \phi_{t_0}^{-1})V(x_{t_0}), V(x_t)).$$
Then:
\begin{itemize}
\item because of the  invariance by symplectic diffeomorphisms, the signature of ${\cal Q}_{K(x_t)}(D(\phi_t\circ \phi_{t_0}^{-1})L_{t_0}, D(\phi_t\circ \phi_{t_0}^{-1})V(x_{t_0}))$ is equal to the signature of ${\cal Q}_{D(\phi_{t_0}\circ \phi_{t}^{-1})^{-1}K(x_t)}(L_{t_0}, V(x_{t_0}))$ and if we chose $\varepsilon$ small enough, this signature is equal to the signature of ${\cal Q}_{K(x_{t_0})}(L_{t_0}, V(x_{t_0}))$; 
\item  the quadratic form  $${\cal Q}_{K(x_t)}(D(\phi_t\circ \phi_{t_0}^{-1})V(x_{t_0}), V(x_t))=-{\cal Q}_{K(x_t)}( V(x_t), D(\phi_t\circ \phi_{t_0}^{-1})V(x_{t_0}))$$ 
is negative definite 
{because we assume that the isotopy $(\phi_t)$ strictly twists the vertical.}
\end{itemize}
As the index of  the sum of a quadratic form $Q$ and a negative {definite} quadratic form is at least { the sum of the index and the nullity} of $Q$, we deduce that for $t\in (t_0, t_0+\varepsilon)$ the  index of ${\cal Q}_{K(x_t)}(L_t, V(x_t))$ is at least the { sum of $1$}, which is the nullity of ${\cal Q}_{K(x_{t_0})}(L_{t_0},V(x_{t_0}))$, and the index of ${\cal Q}_{K(x_0)}(L_{t_0}, V(x_0))$. { Moreover, as the quadratic form  $-{\cal Q}_{K(x_t)}( V(x_t), D(\phi_t\circ \phi_{t_0}^{-1})V(x_{t_0}))$ 
 is close to $0$ for $\varepsilon$ small enough and because of the continuous dependence  of the eigenvalues on the quadratic form, the  index of ${\cal Q}_{K(x_t)}(L_t, V(x_t))$ is exactly the sum of $1$ and the index of ${\cal Q}_{K(x_0)}(L_{t_0}, V(x_0))$.} A similar argument gives that for $t\in (t_0-\varepsilon, t_0)$, the  index of ${\cal Q}_{K(x_t)}(L_t, V(x_t))$ is {exactly} the index of ${\cal Q}_{K(x_0)}(L_{t_0}, V(x_0))$
 This proves that  $(L_t)_{t\in [\alpha, \beta]}$ intersect $\Sigma_1$ { topologically} transversally and  in the negative sense and that the Maslov index is   non-positive.\\ 
 {Let now $(\phi_t)$ be an isotopy that twists the vertical (with no further assumptions). Consider the Lagrangian path $(D\phi_tL)_{t\in[\alpha,\beta]}$. 
 \begin{claim}
 There exists an isotopy $(\tilde\phi_t)$ of conformally symplectic diffeomorphisms of $\M$ such that
 \begin{itemize}
 \item $(D\tilde\phi_tL)_{t\in[\alpha,\beta]}$ is a smooth perturbation of $(D\phi_tL)_{t\in[\alpha,\beta]}$, and in particular, $\mathrm{MI}((D\phi_tL)_{t\in[\alpha,\beta]})=\mathrm{MI}((D\tilde\phi_tL)_{t\in[\alpha,\beta]})$;
 \item $(\tilde\phi_t)$ is an isotopy 
 {that satisfies the conclusion of Proposition\ref{Ptwistintrinsec} (with  definite quadratic forms) in a neighborhood of $(D\tilde\phi_tL)_{t\in[\alpha,\beta]}$.}
 \end{itemize}
 \end{claim}
 The claim immediately implies that the Maslov index of $(D\phi_tL)_{t\in[\alpha,\beta]}$ is non-positive, as desired.}\\
 {Let us now prove the claim.  
Because $(\phi_t)$ twists the vertical, we deduce from Proposition \ref{Ptwistintrinsec} that
there exists $\varepsilon>0$ such that
\begin{itemize}
    \item $\forall t\in (t_0, t_0+\varepsilon),  Q_{K(x_t)}(D\phi_t\circ \big(D\phi_{t_0}\big)^{-1}V(x_{t_0}), V(x_t))
    $ is a negative semi-definite quadratic form;
    \item $\forall t\in (t_0-\varepsilon, t_0),  Q_{K(x_t)}(D\phi_t\circ \big(D\phi_{t_0}\big)^{-1}V(x_{t_0}), V(x_t))
    $  is a positive semi-definite quadratic form.
    
\end{itemize}
{If the vector field associated to $(\phi_t)$ is written in the chart as $X=(X_q, X_p)$, we deduce from equations \eqref{Ehamlin}   that 
$$\frac{d}{dt}\big(b_t(d_t)^{-1}
\big)=\partial_qX_qb_t(d_t)^{-1}+\partial_pX_q-b_t(d_t)^{-1}\partial_qX_pb_t(d_t)^{-1}-b_t(d_t)^{-1}\partial_pX_p\, .
$$
Because $b_t(d_t)^{-1}$ is the matrix of $Q_{K(x_t)}(D\phi_t\circ \big(D\phi_{t_0}\big)^{-1}V(x_{t_0}), V(x_t))$ that is zero for $t=t_0$, we deduce that   
$$\frac{d}{dt}\big(b_t(d_t)^{-1}
\big)_{|t=t_0}= \partial_pX_q\, .
$$
and then that
 $\partial_pX_q$ is a positive semi-definite quadratic form. We now add to $X$ a small Hamiltonian vector-field $Y$ that is associated to a Hamiltonian $H$ that is strictly convex in the fiber direction. This implies that $\partial_pY_q$ is positive definite and so is $\partial_p(X+Y)$. {By Proposition \ref{Ptermchpvect},}  the isotopy that is associated to $X+Y$ is the wanted isotopy $(\tilde\phi_t)$.}}
\end{proof}

\begin{proposition}\label{PHconvextwist}
 Assume that $a:I\rightarrow \R$ and $H:I\times T^*M\rightarrow \R$ are smooth functions and let us use the notation $H_t(x)=H(t, x)$. We assume  that the Hessian of $H$ restricted to every vertical fiber is positive definite. We define the time-dependent vector field $X_t$ of $T^*M$ by
$$i_{X_t}\omega=dH_t-a(t)\lambda.$$
Then the isotopy defined by $X_t$ is conformally symplectic and  strictly twists the vertical.
\end{proposition} 

\begin{proof}[Proof of Proposition \ref{PHconvextwist}.] For $(t_0, x_0)\in I\times T^*M$, we choose a vertically foliated Darboux  chart $F=(F_1, F_2):{\cal U} \rightarrow \R^d\times \R^d$  such that $F(x_0)=0$ and   $F(\cal U)=(-a, a)^{d}\times (-a, a)^d$.\\
We now work in this chart and denote by $H$ the Hamiltonian in this chart, which has a positive definite Hessian in the $p$ direction. We chose $\varepsilon_0 >0$ such that for all $t\in (t_0-\varepsilon_0, t_0+\varepsilon_0)$,
${\cal G}_t:=(\phi_t\circ\phi_{t_0}^{-1})\big(F^{-1}(\{ 0_{\R^d}\}\times (-\frac{a}{2}, \frac{a}{2})^d)\big) \subset{\cal U}$ and 
 $F({\cal G_t})$\footnote{$F({\cal G_t})$ is Lagrangian.} is the graph of a function $p\mapsto q=dg_t(p)$.
 We deduce from the Hamilton equations that there exists $\varepsilon\in (0, \varepsilon_0)$ such that uniformly for $y\in \ (-\frac{a}{2}, \frac{a}{2})^d$ and $t\in (t_0-\varepsilon, t_0+\varepsilon)\backslash \{t_0\}$ if we use the notation $\phi_t(0, y)=(q_t, p_t)$ then 
 $$D^2g_t(p_t)= (t-t_0)\big(\frac{\partial^2 H}{\partial p^2}(t_0, 0, y)+O(t-t_0)\Big).$$
 This gives the {(strict) }twist property. 
 \end{proof} 
\subsection{Maslov index and symplectic reduction}\label{sym reduction}
 On a cotangent bundle, the Maslov index is invariant by symplectic reduction. The result is due to C.~Viterbo \cite{Vit1987}. For sake of completeness, we recall here Viterbo's proof. 
	
	Let us start by showing the invariance of the Maslov index by symplectic reduction on a symplectic vector space. Let $(V,\omega)$ be a symplectic vector space of dimension $2d$. Denote by $\Lambda(V)$ the set of Lagrangian subspaces in $V$ and, for every subspace  $U\subset V$, by $\Lambda_U(V)$ the set of Lagrangian subspaces $L$ such that $L\cap U=\{0\}$.

Fix $L_0\in\Lambda(V)$\footnote{$L_0$ is the Lagrangian subspace with respect to which we calculate the Maslov index in $(V,\omega)$.}. Let $W\subset V$ be a coisotropic (not Lagrangian) vector subspace such that
\[
W^\perp\subset L_0\subset W\, ,
\]
where $W^\perp$ denotes the symplectic orthogonal with respect to $\omega$. Consider the quotient map
\begin{equation*}
\begin{aligned}
\Pi_{W^\perp}:& \,W\to W/W^\perp\\
&\,v\mapsto [v]\, ,
\end{aligned}
\end{equation*}
where $[v]=[u]$ if and only if $v-u\in W^\perp$. Observe that $\Pi_{W^\perp}$ is a surjective linear map.
Then, the quotient space inherits a symplectic 2-form $\omega_W$ from $\omega$, and $(W/W^\perp,\omega_W)$ is still a symplectic vector space. In particular, for every Lagrangian subspace $L$ of $V$ the image $\Pi_{W^\perp}(L\cap W)$ is still a Lagrangian space in $W/W^\perp$. Denote by \[\mathcal{P}_{W^\perp}:\Lambda(V)\hookrightarrow \Lambda(W/W^\perp)\] 
\[
L\mapsto \Pi_{W^\perp}(L\cap W)\, .
\]
{
 
The following holds.
\begin{claim}
 The map $\mathcal{P}_{W^\perp}$ restricted to $\Lambda_{W^\perp}(V)$ is a submersion.
\end{claim} 
\begin{proof}[Proof of the claim.]
Let us fix $L\in\Lambda_{W^\perp}(V)$ and let $L'\in\Lambda(V)$ be such that $W^\perp\subset L'\subset W$ and $L\cap L'=\{0\}$. The set $U=\{ \tilde L\in\Lambda(V) :\ \tilde L\cap L'=\{0\} \}$ is an open neighbourhood of $L$. If $\tilde L\in U$, then there exists a unique linear map $B=B_{\tilde L}: L\to L'$ such that
$$\tilde L=\{ v+Bv; v\in L\}$$
and $B$ satisfies the symmetry condition 
\begin{equation}\label{Esym}\forall \ell, \ell'\in L, \omega(\ell, B\ell')+\omega(B\ell, \ell')=0.\end{equation}
Moreover, if $B:L\to L'$ satisfies the symmetry condition \eqref{Esym}, then the set $\tilde L=\{ v+Bv; v\in L\}$ is a Lagrangian subspace of $V$ that is transverse to $L'$.\\
We denote by $\cb$ the set of linear maps from $L$ to $L'$ that satisfy the symmetry condition \eqref{Esym}. It is a finite dimensional vector space that is the image of the chart 
$$\tilde L\in U\mapsto B_{\tilde L}\in \cb.$$
Similarly, if $\overline{L}=\P_{W^\bot}(L)$ and $\overline{L}'=\P_{W^\bot}(L')$, the set 
$$V=\{\tilde L\in \Lambda(W/W^\bot); \tilde L\cap \overline{L}'=\{ 0\}\}$$ is an open neighbourhood of $\overline{L}$ in $\Lambda(W/W^\bot)$. The map that associate to every $\tilde L\in V$ the linear map $\overline B_{\tilde L}:\overline{L}\to \overline{L}'$ such that $\tilde L=\{\ell+\overline B_{\tilde L}\ell;\ \ell\in\overline L\}$ is a chart whose image is the finite dimensional vector space $\overline{\cb}$ of linear maps $\overline B:\overline{L}\to \overline{L}'$ such that
$$\forall \ell, \ell'\in \overline{L},\omega_W(\ell, \overline B\ell')+\omega_W(\overline B\ell, \ell')=0.
$$
In these charts, the map $\P_{W^\bot}$ is read
$\Phi: \cb\to\overline{\cb}
$ where
$$\Phi(B)=\Pi_{W^\perp}\circ B|_{L\cap W}\circ (\Pi_{W^\perp}\vert_{L\cap W})^{-1}.$$
Hence $\Phi$ is a linear map. This is then a submersion onto its image that is a linear subspace of $\overline{\cb}$. If we prove that $\Phi(\cb)=\overline{\cb}$, we will deduce that $\P_{W^\bot}$ is a submersion. Thus, let $\overline{B}_0\in \overline{\cb}$ and let $\overline{L}_0$ be the graph of $\overline{B}$. We choose a linear subspace  $L'_1$ of $L'$   that is transverse to $W^\bot$ and define $B_1:L\cap W\to L'_1$ as
$$\forall v\in L\cap W, B_1(v)=\Pi_{W^\perp}|_{L'_1}^{-1}\circ \overline{B}_0\circ \Pi_{W^\perp}(v).$$
When $v, w\in L\cap W$, we have  
\begin{eqnarray*}\omega(v, B_1w)+\omega(B_1v, w)&=& 
\omega(v,\Pi_{W^\perp}|_{L'_1}^{-1}\circ \overline{B}_0[w])
+\omega(\Pi_{W^\perp}|_{L'_1}^{-1}\circ \overline{B}_0[v],w)\\
&=&\omega_W([v],  \overline{B}_0[w])
+\omega_W(  \overline{B}_0[v],[w])=0\, .\\
\end{eqnarray*}
Hence $B_1:L\cap W\to L'_1$ satisfies the symmetry condition and then its graph $L_2$ is an isotropic subspace of $(L\cap W)+L'\subset W$ such that $L_2\cap L'=\{ 0\}$. We can choose a Lagrangian subspace $\tilde L$ of $V$ that contains $L_2$ and is transverse to $L'$: $\tilde L$ is then the graph of a map $B_2:L\to L'$ that satisfies the symmetry condition and contains $L_2$. Then the graph of $\Phi(B_2)$ is a Lagrangian subspace of $W/W^\bot$ that contains $\overline{L}_0$, hence is equal to $\overline{L}_0$ and we deduce that $\Phi(B_2)=\overline{B}_0$ and $\Phi$ is surjective.
\end{proof}}


Denote by $i:\Lambda_{W^\perp}(V)\hookrightarrow\Lambda(V)$ the standard inclusion. 
This is 
a submersion. 

\begin{lemma}\label{sym red for vec space}
    Let $t\in [0,1]\mapsto \gamma(t)\in\Lambda_{W^\perp}(V)$ be an arc such that $\gamma(0)\cap L_0=\gamma(1)\cap L_0=\{0\}$. Then
    \[
    \mathrm{MI}(i\circ\gamma)=\mathrm{MI}(\mathcal{P}_{W^\perp}\circ\gamma)\, ,
    \]
    {where 
    \begin{itemize}
        \item the Maslov index $ \mathrm{MI}(i\circ\gamma)$ is calculated with respect to $L_0$ in $\Lambda(W)$;
        \item the Maslov index $ \mathrm{MI}(\mathcal{P}_{W^\perp}\circ\gamma)$ is calculated with respect to $\mathcal{P}_{W^\perp}(L_0)$ in $\Lambda(W/W^\perp)$.
    \end{itemize}}
\end{lemma}
\begin{proof}
 Up to slightly perturb the path, we can assume that $\gamma$ is in D-general position with respect to $\Sigma:=\{L\in\Lambda(V) ;\ L\cap L_0\neq \{0\}\}$.
{The subspace  $L_0':= \mathcal{P}_{W^\perp}(L_0)$ of $W/W^\perp$ is  Lagrangian and we denote $\overline{\Sigma}=\{L'\in\Lambda(W/W^\perp) ;\ L'\cap L_0'\neq\{0\}\}$. Observe that $\mathcal{P}_{W^\perp}^{-1}(\overline{\Sigma})\subset\Sigma$ because $W^\bot\subset L_0$.

Since the maps $i$ and $\mathcal{P}_{W^\perp}$ are submersions, we have
\begin{itemize}
    \item the path $i\circ\gamma$  is in D-general position with respect to $\Sigma$;
    \item the path $\mathcal{P}_{W^\perp}\circ\gamma$ is in D-general position with respect to $\overline{\Sigma}$.
  
\end{itemize}}
 Moreover, the choice of a coorientation of $\Sigma$ determines a coorientation both on $i^{-1}(\Sigma_1)$ and on $\mathcal{P}_{W^\perp}\circ i^{-1}(\Sigma_1)$.
 Following \cite{Vit1987}, we claim that

\begin{claim}\label{claim23}
 \[
 i^{-1}({\Sigma)=
  \mathcal{P}_{W^\perp}^{-1}(\overline{\Sigma}) \cap \Lambda_{W^\bot}(V)}\, .
 \]\end{claim}
 \begin{proof}[Proof of the claim.]We first observe that on one side
 \[
  i^{-1}({\Sigma)}=
  \{ L\in\Lambda(V) :\ L\cap W^\perp=\{0\}\text{ and } L\cap L_0\neq\{0\} \}\,.
 \]
On the other side we have
 \[
 \mathcal{P}_{W^\perp}^{-1}({\overline{\Sigma}}){\cap \Lambda_{W^\bot}(V)}=
 \]
 \[
 \{ L\in{ \Lambda_{W^\bot}(V)} ; ((L\cap W+W^\perp)/W^\perp)\cap ((L_0\cap W+W^\perp)/W^\perp)\neq \{0\} \}\, ;
 \]
since $W^\perp\subset L_0\subset W$, {we have $L_0=L_0\cap W+W^\perp$} and  any $L\in \mathcal{P}_{W^\perp}^{-1}({\overline{\Sigma}} ){\cap \Lambda_{W^\bot}(V)}$ is so that\\
 $
 ((L\cap W+W^\perp)/W^\perp)\cap ({L_0 }/W^\perp)=$\hglue 5 truecm\[ (L\cap L_0\cap W+W^\perp)/W^\perp=(L\cap L_0+W^\perp)/W^\perp\neq \{0\}\, .
 \]
 Since $L\cap W^\perp=\{0\}$
 \[
 (L\cap L_0+W^\perp)/W^\perp\neq \{0\}\quad\Leftrightarrow\quad L\cap L_0\neq \{0\}\, .
 \]
 We so conclude that
\begin{eqnarray*}
 \mathcal{P}_{W^\perp}^{-1}({\overline{\Sigma}}){\cap \Lambda_{W^\bot}(V)}&=&
 \{
 L\in \Lambda(V) :\ L\cap W^\perp=\{0\}\text{ and } L\cap L_0\neq \{0\}
 \}\\
 &=&i^{-1}(\{ L\in\Lambda(V) :\ L\cap L_0\neq\{0\}\})\, .
\end{eqnarray*}
\end{proof}
 
Since $i$ is a submersion, the number of crossings of $\gamma$ with $ i^{-1}({ \Sigma})$ is equal to the number of crossings of $i\circ\gamma$ with ${ \Sigma}$. {Since also $\mathcal{P}_{W^\perp}$ is a submersion and since $\mathcal{P}_{W^\perp}(\Sigma)=\bar\Sigma$}, 
we conclude that the number of crossings of $\gamma$ with $ i^{-1}({ \Sigma})$ is equal to the number of crossings of $\mathcal{P}_{W^\perp}\circ\gamma$ with ${ \overline{\Sigma}}$.
 
Since the coorientation on $i^{-1}(\Sigma)$ and on $\mathcal{P}_{W^\perp}\circ i^{-1}(\Sigma)$ is determined by the coorientation of $\Sigma$, we conclude that actually the number of positive (resp. negative) crossings of $i\circ\gamma$ corresponds to the number of positive (resp. negative) crossings of $\mathcal{P}_{W^\perp}\circ\gamma$. By the definition of Maslov index, we obtain the sought result.
\end{proof}

{We want now to prove the invariance of the Maslov index by symplectic reduction on the cotangent bundle $\M$, endowed with the symplectic form $\omega$. Let $\cv$ be the lagrangian foliation whose fibers of the associated tangent Lagrangian bundle are the vertical Lagrangian subspaces. Let $\cal W\subset \M$ be a coisotropic submanifold {and let $i_{\cal W}:\cal W\to\M$ be the canonical injection}; the characteristic foliation of $\cal W$, { denoted by $\cal W^\bot$,} { admits for tangent bundle} 
${T_x(\cal W^\bot)=}\mathrm{ker}(i^*_{\cal W}\omega)(x)=(T_x\cal W)^\perp$.

Assume that, for every $x\in { \cal W}$ it holds
\[
(T_x\cal W)^\perp\subset T_x\cv\subset T_x\cal W\, .
\]
{We assume that the symplectic reduction of $\cal W$ is a true symplectic manifold that we denote by $\cal R: \cal W\to \cal W/\cal W^\bot$. When $x\in\cal W$ and $L\in \Lambda_{T_x\cal W^\bot}(T_x\M)$, we denote
$$\cal P(L)=D\cal R(x)L=(L+T_x\cal W^\bot)/T_x\cal W^\bot\in \Lambda(\cal W/\cal W^\bot).$$
Then $\cal P$ is a submersion 
 from $\Lambda_{\cal W^\bot}(\M)_{|\cal W}$ to $\Lambda({\cal W}/\cal W^\bot)$.
}

We denote  $\Sigma(\M):= \{ L\in \Lambda(\M) ;\ L\cap T_{p(L)}\cv\neq \{0\} \}$ {
and $\Sigma(\cal W/\cal W^\bot)=\{ L\in \Lambda(\cal W/\cal W^\bot); L\cap T_{p(L)}\cal P(\cal V)\neq \{ 0\}\}$
}.  
\begin{lemma}\label{lemma almost Viterbo}
	Let $\Gamma:[a,b]\to \Lambda(\M)$ be a smooth arc such that 
	\begin{itemize}
		\item $\Gamma(a),\Gamma(b)\notin \Sigma(\M)$;
		\item {$\Gamma$ is in $D$-general position with respect to the fibered singular cycle $\Sigma(\M)$;}
		\item at every point the path has trivial intersection with the tangent bundle of the characteristic foliation of $\cal W$, i.e. 
	\[
	\Gamma(t)\cap (T_{p(\Gamma(t))}\cal W)^\perp=\{0\}\quad\forall t\in[a,b]\, .
	\]
	\end{itemize}
Then
\[
\mathrm{MI}(\Gamma)= \mathrm{MI}(\cal P\circ\Gamma)\, ,
\]
where 
    \begin{itemize}
        \item the Maslov index $ \mathrm{MI}(\Gamma)$ is calculated with respect to $T\cal V$ in $\Lambda(\M)$;
        \item the Maslov index $ \mathrm{MI}(\cal P\circ\Gamma)$ is calculated with respect to $\mathcal{P}(\cal V)$ in {$
        \Lambda(\cal W/\cal W^\bot)$.}
    \end{itemize}
\end{lemma}

\begin{proof}
	 Since $\cal P$ is a submersion { and $\cal P^{-1}(\Sigma(\cal W/\cal W^\bot))\cap \Lambda_{T\cal W^\bot}(\M)=\Sigma(\M)_{|\cal W}$}, also the path $\cal P\circ\Gamma$ is in {$D$-general position} with respect to { $\Sigma(\cal W/\cal W^\bot)$}. 
	 In order to conclude, it is then sufficient to calculate the Maslov index of a sub-path of $\Gamma$, around a (isolated) crossing $t$. Let $U\subset\M$ be a neighborhood of $p\circ\Gamma(t)$ and let 
	\[
	\Gamma\vert_{[t-\epsilon,t+\epsilon]}:[t-\epsilon,t+\epsilon]\to \Lambda(U)\, ,
	\]
	be a Lagrangian path with only an isolated, transverse crossing at $t$. 
	Let us trivialise $\Lambda(U)$ as $U\times \Lambda(T_{p\circ\Gamma(t)}\M)$. Similarly, trivialise the image $\mathcal{P}(\Lambda(U))$ as $\mathcal{P}(U)\times \Lambda(T_{p\circ\Gamma(t)}\mathcal{W}/( T_{p\circ\Gamma(t)}\mathcal{W} )^\perp)$. Up to restrict the neighborhood $U$, the Maslov index of the path $\Gamma\vert_{[t-\epsilon,t+\epsilon]}$ with respect to $\cv$ corresponds to the Maslov index of $\Gamma\vert_{[t-\epsilon,t+\epsilon]}$, seen as a Lagrangian path in the symplectic vector space $T_{p\circ\Gamma(t)}\M$ thanks to the trivialization, with respect to $T_{p\circ\Gamma(t)}\cv$. Similarly, the Maslov index of the path $\mathcal{P}(\Gamma\vert_{[t-\epsilon,t+\epsilon]})$ with respect to $\mathcal{P}(\cv)$ is actually the Maslov index of the Lagrangian path $\mathcal{P}(\Gamma\vert_{[t-\epsilon,t+\epsilon]})$, seen in $T_{p\circ\Gamma(t)}\mathcal{W}/( T_{p\circ\Gamma(t)}\mathcal{W} )^\perp$ through the trivialization, with respect to $\mathcal{P}(T_{p\circ\Gamma(t)}\cv)$.
%
Applying then Lemma \ref{sym red for vec space}, we conclude.
\end{proof}}

\section{Maslov index along a Lagrangian submanifold that admits a generating function}\label{section gfqi}
Let $\cl\subset T^*M$ be a Lagrangian submanifold. The goal of this Section is to prove that every arc $\Gamma:[a,b]\to T\cl$ whose endpoints project on $T^*M$ on the so-called graph selector of $\cl$ has zero Maslov index.
\subsection{The relation  between the Maslov index and the  Morse index} 
Let us recall the definition of generating function for a Lagrangian submanifold $\cl$ of $T^*M$.
{ \begin{defi}
A $C^r$ function with $r\geq 2$ 
function $S:M\times \R^k\rightarrow \R$ {\sl generates} a Lagrangian submanifold $\cl$ of $T^*M$  if
\begin{itemize}
\item using the notation
$$\cC_S=\Big\{ (q, \xi)\in M\times \R^k :\  \frac{\partial S}{\partial \xi}(q, \xi)=0\Big\},$$
at every point of $\cC_S$, the map $\frac{\partial S}{\partial \xi}$ is a submersion; in this case, $\cC_S$ is a $d$-dimensional submanifold of $M\times \R^k$;
\item the map $j_S: \cC_S\hookrightarrow T^*M$ defined by $j_S(q, \xi)=\frac{\partial S}{\partial q}(q, \xi)$ is an embedding such that $j_S(\cC_S)=\cl$.
\end{itemize}
The {\sl generating function} $S$ is {\sl quadratic at infinity} (GFQI) if there exists a compact subset $K\subset M\times \R^k$ and a non-degenerate quadratic form $Q:\R^k\rightarrow \R$ such that 
$$\forall (q, \xi)\notin K, S(q, \xi)=Q(\xi).$$
{The generating function quadratic at infinity $S$ is of index $m$ if the non-degenrate quadratic form $Q$ has index $m$.}
\end{defi}

A result due to Sikorav \cite{Bru1991,Sik1987}, asserts that every H-isotopic\footnote{This means Hamiltonianly isotopic} to the zero section submanifold of $T^*M$ admits a GFQI.
}
 \begin{nota} If we denote as before the Liouville form on $T^*M$ by $\lambda$ and the Liouville form on $T^*(\R^k)$ by $\lambda_1$, the  product manifold $\cN=T^*M\times T^*(\R^k)$ is endowed with the symplectic form 
$\Omega=-p_1^*d\lambda-p_2^*d\lambda_1$ where $p_i$ is the projection on the i-th factor.

\end{nota}
\begin{theorem}\label{Tmain}
Let $\cl\subset T^*M$ be a Lagrangian submanifold that admits a  generating function $S(q, \xi):M\times \R^k\rightarrow \R$. Let $(q_i, \xi_i)\in M\times \R^k$, $i= 1, 2$ be such that
\begin{itemize}
\item $\frac{\partial S}{\partial \xi}(q_i, \xi_i)=0$, i.e. $(q_i,\xi_i)\in \cC_S$;
\item if we use the notation $p_i= \frac{\partial S}{\partial q}(q_i, \xi_i)$, the submanifold $\cl$ is transverse to the vertical fiber $T^*_qM$ at $p_i$ in $T^*M$.
\end{itemize}
Then, $\ker \frac{\partial^2 S}{\partial \xi^2}(q_i, \xi_i)=\{ 0\}$ and for every arc $\gamma_0$ joining  $\gamma_0(0)=p_1$ to  $\gamma_0 (1)=p_2$ in $\cl$, the Maslov index of $t\in [0, 1]\mapsto T_{\gamma_0 (t)}\cl$ with respect to the vertical is equal to the difference of the Morse indices $\text{index}\big(\frac{\partial^2 S}{\partial \xi^2}(q_2, \xi_2)\big)-\text{index}\big(\frac{\partial^2 S}{\partial \xi^2}(q_1, \xi_1)\big)$.

\end{theorem}
 
\begin{proof}[Proof of Theorem \ref{Tmain}] 
\begin{lemma}\label{Lcaractvert}
Let $p=\frac{\partial S}{\partial q}(q, \xi)\in \cl$. Then $\cl$ is transverse to $T^*_qM$ at $p$ if and only if  $ker \big( \frac{\partial^2 S}{\partial \xi^2}(q, \xi)\Big)=\{ 0\}$.\end{lemma}
\begin{proof}[Proof of Lemma \ref{Lcaractvert}]
Let us fix $p\in \cl$ and let $\delta p\in T_p(T^*M)$. We use the notation $q=\pi(p)\in M$ and $\delta q= d\pi(p)\delta p\in T_qM$.\\
Then  $\delta p$ belongs to $T_{p}\cl$ if and only if there exists $\delta\xi\in \R^k$ such that 
\begin{itemize}
\item ${D\big(\frac{\partial S}{\partial \xi}\big)(\delta q, \delta \xi)=}\frac{\partial^2 S}{\partial q\partial \xi}(q, \xi)\delta q+\frac{\partial^2 S}{\partial \xi^2}(q, \xi)\delta \xi=0$;
\item ${ Dj_S(\delta q, \delta \xi)=}\delta p=\frac{\partial^2 S}{\partial q^2}(q, \xi)\delta q+\frac{\partial^2 S}{\partial \xi\partial q}(q, \xi)\delta \xi$.
\end{itemize}
{Observe that $\pi\Big(\frac{\partial S}{\partial q}(q, \xi)\Big)=q$ and then
\begin{equation}\label{Egenevert} D\pi\Big(\frac{\partial^2S}{\partial q^2}(q, \xi)\delta q\Big)=\delta q\quad{\rm and}\quad D\pi\Big(\frac{\partial^2S}{\partial \xi\partial q}(q, \xi)\delta \xi\Big)=0.\end{equation}
We deduce 
 $$\delta\xi\in\ker \Big(\frac{\partial^2S}{\partial\xi^2}\Big)\backslash\{ 0\}\Longleftrightarrow(0, \delta \xi)\in \ker\Big(D\big(\frac{\partial S}{\partial \xi}\big)\Big)\backslash\{ 0\}$$
$$\Longleftrightarrow Dj_S(0, \delta \xi)\in T\cl\backslash\{ 0\}\Longleftrightarrow\frac{\partial^2S}{\partial \xi\partial q}\delta\xi\in T\cl\backslash\{0\}.$$Using \eqref{Egenevert}, we conclude that $\ker \Big( \frac{\partial^2 S}{\partial \xi^2}(q, \xi)\Big)\neq\{ 0\}$ if and only if $\cl$ is not transverse to $T^*_qM$ at $\frac{\partial S}{\partial q}(q, \xi)$.}
\end{proof}

In ${\cal N}=T^*M\times T^*(\R^k)$,  endowed with the symplectic  form $\Omega${$=-p_1^*d\lambda-p_2^*d\lambda_1$}, we consider the coisotropic foliation into submanifolds $$\cw_\chi=T^*M\times \R^k\times \{\chi\}$$ 
for $\chi\in \R^k$.
The characteristic leaves of $\cw_\chi$ are the submanifolds   $\cw^\bot_{( p,\chi)}=\{ p\}\times \R^k\times \{\chi\}$ with $ p\in T^*M$.

We will use  also the  Lagrangian foliation $\cF$  of   $\cN$ with leaves 
$\cF_{q,\chi}=T^*_qM\times \R^k\times \{ \chi\}$. Then we have  $\cw^\bot_{( p,\chi)} \subset \cF_{(\pi(p),\chi )}\subset \cw_\chi$. We denote by $F_{(p,\xi, \chi)}$ the tangent space to the leaf $\cF_{(\pi(p),\chi)}$ at the point  $(p,\xi, \chi)$.\\
The graph $ \cg= \text{graph}(dS)\subset \cN$ of $dS$ is  a Lagrangian submanifold of $\cN$  that is transverse to $\cw_0$ and such that $\cg\cap \cw_0$ is diffeomorphic to $\cl$ by the map $$\cR:(p, \xi, 0)\in \cw_0\mapsto p.$$ Observe that $\cR$ is the symplectic reduction of $\cw_0$. We denote by $R$ the restriction of $\cR$ to $\cg\cap \cw_0$.

We use for $\gamma_0$, $q_i, p_i, \xi_i$ the same notations as in Theorem \ref{Tmain}. Then $\Gamma_0=R^{-1}\circ \gamma_0$ is an arc on $\cg\cap \cw_0$ such that $\Gamma_0(0)=(p_1, \xi_1, 0)$ and $\Gamma_0 (1)=(p_2, \xi_2, 0)$. We have
\begin{lemma}\label{Lindices}
Let   $\Gamma(t)=( p(t), \xi(t), \chi(t))\in \cg$ be an arc in $\cg$ such that at  $\Gamma(0)$ and $\Gamma(1)$, the quadratic form $\frac{\partial^2 S}{\partial \xi^2}$ is non-degenerate. The Maslov index of the arc of Lagrangian subspaces $t\in[0, 1]\mapsto T_{\Gamma(t)}\cg$ with respect to the fibered singular cycle associated to $\cF$ 
is $$\text{index}\Big(\frac{\partial^2 S}{\partial \xi^2}(q(1), \xi(1))\Big)-\text{index}\Big(\frac{\partial^2 S}{\partial \xi^2}(q(0), \xi(0))\Big).$$
\end{lemma}
\begin{proof}[Proof of Lemma \ref{Lindices}] 
{\color{black}Up to a small perturbation, there is no loss of generality in assuming that $S$ is smooth. 
The proof is divided into two steps. {   First of all, we will perturb the Lagrangian submanifold $\cl$ (i.e., its generating function) and $\Gamma$ on it in such a way   that $\Gamma$ is in $D$-general position 
with respect to the fibered singular cycle associated to $\cF$. Then we will prove the lemma. }\\}

\noindent\textit{First step.}  {\color{black} As $T_{\Gamma(0)}\cg$ and $T_{\Gamma(1)}\cg$ are transverse to $F_{\Gamma(0)},F_{\Gamma(1)}$ respectively, there exists $\varepsilon>0$ such that  for all $ t\in [0, \varepsilon]\cup[1-\varepsilon, 1],  T_{\Gamma(t)}\cg$ is transverse to $F_{\Gamma(t)}$. We use the notation   $t\mapsto \zeta(t)$ $ := (\pi\circ p(t), \xi(t))\in M\times \R^k$. We now choose a neighbourhood $\cal U$ of   {$\zeta([\varepsilon, 1-\varepsilon])$} in $M\times \R^k$ and a diffeomorphism $\psi: \cal U\rightarrow \R^d\times \R^k$ such that 
$$\forall t\in [\epsilon,1-\epsilon], 
{\psi(\zeta (t))}=(t, 0, \dots, 0).$$ 
Let $s:\psi({\cal U})\rightarrow \R$ defined by $s(y)=S\circ \psi^{-1}(y )$. Then in the neighborhood of $(\varepsilon, 0\dots, 0)$ and $(1-\varepsilon, 0, \dots, 0)$, we  know that ${\rm graph}(D^2s)$ and $D\psi(F)$ are transverse.
We can slightly perturb the path of matrices $t\in [\varepsilon, 1-\varepsilon]\mapsto D^2s(t, 0, \dots, 0)$ in a path $t\mapsto A(t)$ of symmetric matrices such that \begin{itemize}
\item $A(t)=D^2s(t, 0, \dots, 0)$ in a neighbourhood of $\varepsilon$ and $1-\varepsilon$;
\item  the path $t\mapsto {\rm graph}(A(t))$ is in $D$-general position
\end{itemize}
We now define for $x_1$ in a neighborhood of $[\varepsilon, 1-\varepsilon]$
\begin{itemize}
\item $\delta(x_1)=\int_\varepsilon^{x_1}(A(\sigma)-D^2s(\sigma, 0, \dots, 0))(1, 0, \dots, 0)d\sigma$;
\item $v(x_1)=\int_\varepsilon ^{x_1}\delta(\sigma)(1,0,\dots,0)d\sigma$;
\item in a neighbourhood of $[\varepsilon, 1-\varepsilon]\times \{ 0_{\R^{d-1}\times \R^k}\}$, 
$$u(x_1, \dots, x_{d+k})=s(x_1, \dots, x_{d+k})+v(x_1)+\delta(x_1)(0,x_2, \dots, x_{d+k})$$
$$+\frac{1}{2}(A(x_1)-D^2s(x_1,0,\dots,0))((0, x_2, \dots, x_{d+k}), (0, x_2, \dots, x_{d+k})).$$
\end{itemize}
Then $u$ is $C^2$ close to $s$ and we have 
$$\forall x_1\in [\varepsilon, 1-\varepsilon], D^2u(x_1, 0,\dots, 0)=A(x_1).$$
We then use a bump bunction $\eta$ with support in a neighbourhood of $[\varepsilon, 1-\varepsilon]\times \{ 0_{\R^{d-1}\times \R^k}\}$ and that is equal to 1 in a smaller neighbourhood of $[\varepsilon, 1-\varepsilon]\times \{ 0_{\R^{d-1}\times \R^k}\}$. We define
$$\tilde s(x_1, \dots, x_{d+k})=$$
$$(1-\eta(x_1, \dots, x_{d+k}))s(x_1, \dots, x_{d+k})+\eta(x_1, \dots, x_{d+k}) u(x_1, \dots, x_{d+k}).$$
As $\tilde s$ is equal to $u$ in $[\varepsilon, 1-\varepsilon]\times \{ 0_{\R^{d-1}\times \R^k}\}$, $D^2\tilde s$ is in {$D$-general position} with respect to $D\psi(F)$ along the lift of  this arc in ${\rm graph }Du$. In addition, as $\tilde s$ is $C^2$-close to $s$, $D^2\tilde s$ is transverse to $D\psi(F)$ along the lift of  $\big( [0, \varepsilon]\cup[ 1-\varepsilon, 1]\big)\times \{ 0_{\R^{d-1}\times \R^k}\}$ in ${\rm graph}D\tilde s$.

 Finally, define the function $\tilde S$ to be equal to $S$ outside $\psi^{-1}({\cal U})$ and to $\tilde s\circ \psi$ in $\psi^{-1}({\cal U})$. Thus, $\tilde S$ is $C^2$ close to $S$,
 {$D\tilde S\circ \zeta$ is $C^1$  close to $DS\circ\zeta=\Gamma$ and $t\mapsto D\tilde S\circ\zeta(t)$ is in {$D$-general position} with respect to the fibered singular cycle associated to $\cF$}. As the new generating function 
 $\tilde S$ is $C^2$  close to $S$,  the number   $\text{index}\big(\frac{\partial^2 S}{\partial \xi^2}(q(1), \xi(1))\big)-\text{index}\big(\frac{\partial^2 S}{\partial \xi^2}(q(0), \xi(0))\big)$ does not change. }

This will allow us to assume that 
 the path $t\mapsto T_{\Gamma(t)}\cg$ is in {$D$-general position} with respect to  {the fibered singular cycle associated to $\cF$}.\\
 
 \noindent \textit{Second step.} We then look at what happens at a crossing $\Gamma(\bar t)$. We choose a chart close to $\frac{\partial S}{\partial q}(q(\bar t), \xi(\bar t))$ and  we assume that we work in coordinates~: $q\in U\subset \R^n$ and $(p_1, \dots, p_n)$ are the dual coordinates defined by $(q, \sum p_idq_i)\in T^*U$.\\
In these coordinates for $(q, p)=(q, \frac{\partial S}{\partial q}(q, \xi))\in U\times \R^n$, we use the linear and symplectic change of coordinates
$$(\delta q, \delta p)\mapsto \Big(\delta Q=\delta p+\big({\bf 1}_n-\frac{\partial^2 S}{\partial q^2}(q, \xi)\big)\delta q, \delta P=-\delta q\Big).$$
In the extended space $T\cN$ with linear  coordinates $(\delta Q, \delta P, \delta \xi, \delta \chi)$, the equation of $F_{\Gamma(t)}$ is $(\delta P, \delta\chi)=(0, 0)$, i.e. $F_{\Gamma(t)}$ is the graph the zero function. The equation of $T\cg$ is 
\begin{equation}\label{ETG}
\begin{cases}\delta P=-\delta Q+\frac{\partial^2S}{\partial \xi\partial q}\delta \xi\\
 \delta \chi=\frac{\partial^2 S}{\partial q\partial \xi} \delta Q+\Big(\frac{\partial ^2S}{\partial \xi^2} -\frac{\partial^2 S}{\partial q\partial \xi} \frac{\partial^2S}{\partial\xi\partial q} \Big)\delta \xi\end{cases}\end{equation} 
 and this is also a graph. 
 We compute then  the change of Maslov index with respect to $F_{\Gamma(t)}$
 with the help of the height  of $T\cg$ above the vertical $L$ with respect to $F$, {i.e. $\cal Q_F(L,T\cal G)$,} where $L$ has   equation  $(\delta Q, \delta\xi)=(0, 0)$. {Observe that, for $t$ close to $\bar t$}, $F$ and $L$ are transverse and the projection $p_F:T\cN \rightarrow  T\cN/F$ restricted to $L$  is an isomorphism. So we can take $(\delta P, \delta \chi)$ as   coordinates in $T\cN/F$.  Also, for $t\not=\bar t$ close to $\bar t$, $F$ and $T\cg$ are transverse, {{because crossings of a path in general position are isolated}}. {Moreover, note that, for $t$ close to $\bar t$, $L$ and $T\cg$ are transverse}. If we introduce the matrix 
 
 \begin{equation}M(t)=\begin{pmatrix} -{\bf 1}_n & \frac{\partial^2 S}{\partial\xi\partial q}\\

\frac{\partial^2 S}{\partial q\partial \xi} & \left(\frac{\partial ^2S}{\partial \xi^2}-\frac{\partial^2 S}{\partial q\partial \xi}\frac{\partial^2S}{\partial\xi\partial q}\right)
\end{pmatrix}(q(t), \xi(t)),\end{equation} 
as the equation of $T\cg$ is (we write in coordinates)
$$\begin{pmatrix}
\delta P \\ \delta \chi
\end{pmatrix} = M(t)\, \begin{pmatrix}
\delta Q \\ \delta \xi
\end{pmatrix}\, ,$$
 the matrix $M(t)$ is invertible for $t\not=\bar t$ and
we have 
 
$$\cq_F(L, T\cg)(\delta P, \delta\chi)=\Omega\Big((0,0,\delta P,  \delta \chi), ((\delta P,\delta\chi) .(M(t)^{-1})^T, \delta P, \delta\chi)\Big).$$
\color{black}
Hence the matrix of $\cq_F(L, T\cg)(\delta P, \delta\chi)$  in coordinates $(\delta P, \delta\chi)$ is $M(t)^{-1}$. The change of signature of $M(t)^{-1}$ at $\bar t$ is exactly the same as the change of signature of $M(t)$.

Let us introduce the matrix
$$\P(t)=\begin{pmatrix} {\bf 1}_n& \frac{\partial^2 S}{\partial\xi\partial q}(q(t), \xi(t))\\  {\bf 0}& {\bf 1}_k\end{pmatrix}.$$
Then we have 
$$\begin{matrix}\P(t)^T M(t)\P(t)&=\begin{pmatrix} {\bf 1}_n& {\bf 0}\\ \frac{\partial^2 S}{\partial q\partial \xi}(q(t), \xi(t))& {\bf 1}_k\end{pmatrix}M(t)\begin{pmatrix} {\bf 1}_n& \frac{\partial^2 S}{\partial\xi\partial q}(q(t), \xi(t))\\  {\bf 0}& {\bf 1}_k\end{pmatrix}\\
&=\begin{pmatrix} -{\bf 1}_n & 0\\ 0&\frac{\partial ^2S}{\partial \xi^2}(q(t), \xi(t)) \end{pmatrix}\, .\hfill\ \end{matrix}$$
Hence the change of signature of $M(t)$ at $t=\bar t$ along the path $\Gamma$ is equal to the change of signature of $\frac{\partial^2 S}{\partial\xi^2}$. This is exactly the Maslov index of the arc of Lagrangian subspaces {\color{black} $t\in[\bar t-\varepsilon, \bar t+\varepsilon]\mapsto T_{\Gamma(t)}\cg$} with respect to $F_{\Gamma(t)}$.
 \end{proof}
We deduce that the Maslov index of $T\cg$ along the arc $\Gamma_0$ with respect to $F$ 
is $\text{index}\big(\frac{\partial^2 S}{\partial \xi^2}(q_2, \xi_2)\big)-\text{index}\big(\frac{\partial^2 S}{\partial \xi^2}(q_1, \xi_1)\big)$.\\
We have noticed that $\cw^\bot_{(p,\chi)} \subset \cF_{(\pi(p),\chi )}\subset \cw_\chi$. Also, because $\cg$ is Lagrangian and transverse to $\cw_0$, at every point of intersection, the intersection of the tangent subspaces to $\cg$ and $\cw^\bot_{( p,0)}$ is $\{ 0\}$. {The path $t\mapsto \Gamma(t)$ can be put in $D$-general position with respect to $\cF$, as done in the first step of the proof of Lemma \ref{Lindices}. Thus, we  can apply the results concerning the Maslov index that are given in section 2 of \cite{Vit1987}, see here Lemma \ref{lemma almost Viterbo} and Subsection \ref{sym reduction}}. \\
As the curve $\Gamma_0$ is contained in $\cg\cap \cw_0$, the Maslov index of $T\cg$  along $\Gamma_0$ with respect $F$ is equal to the Maslov index of $(T(\cg\cap \cw_0))/T\cw^\bot_0$ with respect to $F/T\cw^\bot_0$. We have $(T(\cg\cap \cw_0))/T\cw^\bot_0={ T}R(\cg)={ T}\cl$ and $F/T\cw^\bot_0(\Gamma(t))=T_{\gamma_0(t))}(T^*_{\pi\circ \gamma_0(t)}M)$ is the vertical {$V_{\gamma_0(t)}$}. This proves the theorem.

\end{proof}
 \subsection{Maslov index along graph selectors}\label{SMaslocalongselec}

Let us assume that the Lagrangian submanifold $\cl$ of $T^*M$ admits a 
generating function {quadratic at infinity}. We recall the construction of a    {\sl graph selector} 
$u:M\rightarrow \R$. Such a graph selector was introduced by M.~Chaperon in \cite{Chap1991} (see \cite{PPS2003} and \cite{Sib2004} too) by using the homology. Here we will use  the cohomological approach (see e.g. \cite{ArnaVent2017}). We now explain this.  
\begin{notas}
Let $S: M\times \R^k\rightarrow \R$ be a  function that generates a Lagrangian submanifold, $q\in M$ and $a\in\R$ is a real number, we denote the sublevel with height $a$ at $q$ by $$S^a_q=\{\xi\in\R^k;  \quad S(q, \xi)\leq a\}$$ and we use the notation $S_q=S(q, .)$.
\end{notas}
When $S$ is quadratic at infinity  with index $m$, there exists $N\geq 0$ such that all the critical values of $S$ are in $(-N, N)$. 
Observe that, since $S$ is a  GFQI, $S_q^{-N}$ is the sublevel of a non-degenerate quadratic form of index $m$. Thus {(see for example \cite{Milnor})}, the De Rham relative cohomology space 
	$H^*(\R^k, S_q^{-N})$ is isomorphic to \[
	H^*(\R^m)=\begin{cases}
	\R\quad\text{if }*=m,\\
	0\quad\text{if }*\neq m.
	\end{cases}
	\]
 We denote by $\alpha_q$ a closed $m$-form  of  $\R^k$  such that $\alpha_{q|S_q^{-N}}=0$ and $0\not= [\alpha_q] \in H^m({  \R^k}, S_q^{-N})$.

If $a\in (-N, N)$, we use the notation $i_a: (S_q^a, S_q^{-N})\hookrightarrow (  \R^k, S_q^{-N})$ for the inclusion and then $i_a^*:H^m(  \R^k, S_q^{-N})\rightarrow H^m(S_q^a, S_q^{-N})$. The {\sl graph selector} $u:M\rightarrow \R$ is then defined by:
$$u(q)=\sup\{ a\in \R; [i_a^*\alpha_q]=0\}=\inf\{ a\in \R; [i_a^*\alpha_q]\not=0\}.$$
The following result is classical (see   \cite{ArnaVent2017} for a proof in our setting).

\begin{proposition}\label{geneselec}
Let $\cl\subset T^*M$ be a Lagrangian submanifold admitting a GFQI $S:M\times \R^k\rightarrow \R$ {of regularity $C^r$ with $r\geq 2$} and let $u:M\rightarrow \R$ be the graph selector for $S$ which is a Lipschitz function. Then $u$ is $C^r$ on the open set
\begin{center}
$U := \{ q \in M,\; \xi \mapsto S(q,\xi)$ is Morse excellent\footnote{An \textit{excellent} function is by definition a function whose   every critical value is attained at at exactly one critical point.}$\}$
\end{center}
which has full measure,  and for all $q$ in $U$, the following hold:
\begin{itemize}
\item[$\bullet$] $du(q)\in \cl$;
\item[$\bullet$] $u(q)=S\circ j_S^{-1}(du(q))$.
\end{itemize}


\end{proposition}

\begin{remk}
    Let $\cl$ be a Lagrangian submanifold admitting a generating function $S : M \times \R^k \to \R$. Then for all $C^1$ path $\gamma : [0,1] \to \cl$,
 \[ S\left(j_S^{-1}(\gamma(1))\right)- S\left(j_S^{-1}(\gamma(0))\right) = \int_\gamma \lambda.\]
 As a consequence we may describe the open set $U$ without mentioning the generating family:
 \[\begin{split}
U & =\big\{  q \in M, \, T_q^\star M \pitchfork {\cl}  \text{ and for all path }\gamma:[0,1] \to \cl\\ &
\text{ with distinct endpoints in }T_q^\star M \pitchfork \cl, \quad \int_\gamma \lambda \neq 0.\big\}  \end{split}\]
Indeed, the transversality condition is equivalent to the fact that $(q,\xi) \to S(q,\xi)$ is Morse, and the condition on the path gives that the values of $S$ above two different critical points of the generating family are necessarily distinct.
\end{remk}

From Theorem \ref{Tmain} and the latter proposition, we deduce
\begin{proposition}\label{Cselecindconst}
We use the same notations as in the previous proposition. Then if $q_1, q_2\in U$ and if $\gamma:[0, 1]\rightarrow \cl$ is a continuous arc joining $du(q_1)$ to $du(q_2)$, the Maslov index of the arc of Lagrangian subspaces $t\mapsto T_{\gamma(t)}\cl$ with respect to the vertical is zero.
\end{proposition}
\begin{proof}[Proof of Proposition \ref{Cselecindconst}]
We recall some well-known facts about Morse functions. Let  $f:\R^k\rightarrow \R$ be a Morse function quadratic at infinity 
such that its critical points have different critical value. We use the notation $f^a=\{ x\in\R^k, f(x)\leq a\}$ for the sublevels of $f$.  Then,  for every critical point $c$ such that $D^2f(c)$ has index $p$, for $\varepsilon>0$ small enough, the De Rham relative cohomology space 
$H^*(f^{f(c)+\varepsilon},  f^{f(c)-\varepsilon})$ is isomorphic to $\R$ for $*=p$ and trivial if $*\not=p$. 

Let now consider $q\in U$.  As $S(q, .)$ is Morse such that different critical points correspond to different critical values, there is only one $\xi_q\in \R^k$ that is a critical point of $S(q, .)$ such that $S(q, \xi_q)=u(q)$. By definition of $u$, we have 
\begin{itemize}
\item for every $\varepsilon>0$, $0\not=[i_{u(q)+\varepsilon}^*\alpha_q]\in H^m(S_q^{u(q)+\varepsilon}, S_q^{-N})$;
\item for every $\varepsilon>0$, $0=[i_{u(q)-\varepsilon}^*\alpha_q]\in   H^m(S_q^{u(q)-\varepsilon}, S_q^{-N})$.
\end{itemize}
{We recall the notation for maps of pairs in relative cohomology. The notation $f:(M, N)\to (V, W)$ means that $f:M\to V$ with $f(N)\subset W$.\\
}
We introduce the maps associated to the inclusion $S_q^{-N}\subset S_q^{u(q)-\varepsilon}\subset S_q^{u(q)+\varepsilon}$.  More precisely, we denote by $j_1:(S_q^{u(q)+\varepsilon},S_q^{-N})\hookrightarrow (S_q^{u(q)+\varepsilon},S_q^{u(q)-\varepsilon})$ and $j_2:(S_q^{u(q)-\varepsilon},S_q^{-N})\hookrightarrow (S_q^{u(q)+\varepsilon},S_q^{-N})
$ the two inclusion maps. We now use the exact cohomology sequence that is induced by these maps, see \cite{God1971}, that is
$$H^m (S_q^{u(q)+\varepsilon},S_q^{u(q)-\varepsilon})\stackrel{j_1^*}{\longrightarrow}  H^m (S_q^{u(q)+\varepsilon},S_q^{-N})\stackrel{j_2^*}{\longrightarrow}  H^m (S_q^{u(q)-\varepsilon},S_q^{-N}).
$$
Then $[i_{u(q)+\varepsilon}^*\alpha_q]$ is a non-zero element of $ H^m(S_q^{u(q)+\varepsilon}, S_q^{-N})$ and its image by $j_2^*$ is $0$. Because the sequence is exact, $[i_{u(q)+\varepsilon}^*\alpha_q]$ is a non zero element of the image of $j_1^*$. This implies that $H^m (S_q^{u(q)+\varepsilon},S_q^{u(q)-\varepsilon})\not=\{ 0\}$, and then that  the index of the critical point $\xi_q$ of $S(q, .)$ is $m$.

Hence we have proved that for every $q\in U$, if $du(q)=\frac{\partial S}{\partial q}(q, \xi_q)$ where $\frac{\partial S}{\partial \xi}(q, \xi_q)=0$, the index of $\frac{\partial^2 S}{\partial \xi^2}(q, \xi_q)$ is $m$. We deduce from Theorem \ref{Tmain}  the wanted result.

\end{proof}

\section{Dynamical Maslov index, graph selectors and proof of Theorem~\ref{Tppal}}\label{section heart}

\subsection{Graph selector techniques adapted to conformal symplectic isotopies of the zero section}
\label{exactification}

Let $(\phi_t)$ be a  {$C^{r+1}$}   isotopy  of  conformally symplectic diffeomorphisms  of $T^*M$ { with $r\geq 2$ } such that $\phi_0={\rm Id}_{T^*M}$.  We want to apply the results of section \ref{SMaslocalongselec} to the images $\phi_t(\cl_0)$ of the zero-section and obtain results for the dynamical Maslov index $$\text{DMI}(T_x\cl,(\phi_s)_{s\in[0,t]}).$$

As every $\phi_t$ is conformally symplectic, there exists $a(t)\in]0, +\infty[$ such that $\phi_t^*\omega=a(t)\omega$. Then the form  $\beta_t=\phi_t^*\lambda-a(t)\lambda$ is closed. The projection $\pi: T^*M={\cal M}\rightarrow M$ inducing an isomorphism in cohomology, we can choose {in a $C^{r}$ way }    a closed 1-form $\eta_t$ on $M$ such that $\pi^*\eta_t-\beta_t$ is exact {\color{black} and $\eta_0=0$}. If the symplectic diffeomorphism $f_t:{\cal M}\rightarrow {\cal M}$ is defined by $f_t(p)=p-\eta_t$,   we have 
$$f_t^*\lambda=\lambda-\pi^*\eta_t.$$
If $(\psi_t)$ is the isotopy of conformally symplectic diffeomorphisms defined by $\psi_t=f_t\circ\phi_t $, then we have
$$\psi_t^*\lambda=\phi_t^*(\lambda-\pi^*\eta_t)=a(t)\lambda + \Big( \beta_t-\phi_t^*\pi^*\eta_t\Big).$$
The action of $\phi_t$ on cohomology is trivial because $\phi_t$ is homotopic to ${\rm Id}_{{\cal M}}$. As $\pi^*\eta_t-\beta_t$ is exact, we deduce that $\psi_t^*\lambda-a(t)\lambda$ is exact. 
Hence the image by $\psi_t$ of every H-isotopic to the zero-section submanifold $\cl$ is also H-isotopic to the zero-section, see \cite[Corollary 3]{ArnFej21}. It admits a generating function quadratic at infinity $S_t:M\times\R^k\rightarrow \R$ and a Lipschitz continuous graph selector $u_t:M\rightarrow \R$. 
~\newline

\begin{remk}
The generating function $S_t$ is not unique. For every segment $[a ,b]$ of $\R$, we can choose an integer  $k\in\N$   uniformly in $t\in[a, b]$ and  {in a $C^{r}$ way} a {$C^r$ } generating function $S_t:M\times \R^k\rightarrow \R$ for {${\cal L}_t=\psi_t(\cl)$}. Then the associated graph selector\footnote{It can be proved that up to a constant, $u_t$ is independent of  the chosen generating function $S_t$ of { $\psi_t(\cl)$.}} $u_t$ also  depends {in a $C^{r}$ way}  on $t$.
\end{remk}

{
As in Proposition \ref{geneselec}, we define
$U_t:=\{q \in M, \, \xi \mapsto S_t(q,\xi)$ is Morse excellent$\}$, which is an open set of $M$ with full Lebesgue measure, on which $u_t$ is {$C^r$ } and
\begin{equation}\label{Eselecnonexact}\forall q\in U_t, \quad du_t(q) \in \psi_t(\cl_0) \textit{ i.e. } \eta_t(q)+du_t(q)\in\phi_t(\cl_0).\end{equation}

\begin{proposition}\label{tUouvert} The set ${{\cal U}=}\underset{t}{\bigcup} \,  U_t\times \{t\}$ is an open set of $M \times \R$, on which the function $(q,t) \mapsto u_t(q)$ is {$C^{r}$}. \end{proposition}
\begin{proof}

{The set ${\cal W}=\{ (q, t)\in M\times \R;T_q^\star M \pitchfork {\cl_t}\}  $ is open thanks to Thom transversality theorem. Hence for every $(q_0, t_0)\in {\cal U}$, there exists an open subset $U$ of ${\cal W}$ that contains $(q_0, t_0)$, an integer $N\geq 1$ and $N$ $C^{r-1}$-maps $x_i:U\to T^*M$ such that \begin{itemize}
    \item $\pi\circ x_i(q, t)=q$;
    \item $\forall i\neq j; x_i(q, t)\neq x_j(q;t)$;
    \item $T^*_qM\cap \cl_t=\{ x_1(q, t); \dots ; x_N(q, t)\}$.
    \end{itemize}
As $S$ depends in a $C^r$ way on $(q, t)$, the map 
$Y: {\cal W}\to \R^N$ that is defined by
$$Y(q,t)=\big(S_t\left(j_S^{-1}(x_1(q,t))\right), \dots ,S_t\left(j_S^{-1}(x_N(q,t))\right)\big).$$
is continuous and then 
$$\cU\cap U=\{ (q,t)\in U; \forall i\neq j, S_t\left(j_S^{-1}(x_i(q,t))\right)\neq S_t\left(j_S^{-1}(x_j(q,t))\right)\}
$$
is open because it is the backward image by $Y$ of an open subset of $\R^N$. We have then proved that $\cal U$ is open.
}

Let $q_0 \in U_{t_0}$ for some $t_0$. By definition of the graph selector there exists $\xi_0$ such that $u_{t_0}(q_0)=S_{t_0}(q_0,\xi_0)$ and $\frac{\partial S_{t_0}}{\partial \xi_0}(q_0,\xi_0)=0$. Since $S_{t_0}(q_0,\cdot)$ is Morse, we may apply the implicit function theorem to get a {$C^r$} function $( q,t)\mapsto \xi(q,t)$ solving $\frac{\partial S_t}{\partial \xi}(q,\xi(q,t))=0$ on an open {connected} neighbourhood of $(q_0, t_0)$ {in $\cal U$}
. By continuity of $(t,q)\mapsto u_t(q)$ and since we excluded the case where $S_t(q,\cdot)$ attains a critical value more than once, we also have $u_t(q)=S_t(q,\xi(t,q))$ on this neighbourhood. Thus $(t,q)\mapsto u_t(q)$ is {$C^r$} at $(t_0,q_0)$, hence on the whole set $\cal U$. 
\end{proof}

}


\subsection{Proof of Theorem \ref{Tppal}}

{Let us begin with the case where $\cl$ is the zero section, denoted by $\cl_0$.
With the notations that we introduced in the previous paragraph, we are reduced to prove that $\text{DMI}(T_x\cl_0,(\psi_s)_{s\in[0,t]})=0$ for every $x\in \psi_{t}^{-1}(\text{graph}(du_{t|U_t}))$. This is a result of the two following lemmata for which we provide proofs.}
\begin{lemma}\label{LDMIgraphselec} There exists an integer $n_t$ such that 
$$\forall x\in \psi_{t}^{-1}({\rm graph}(du_{t|U_t}))\qquad  {\rm DMI}(T_x\cl_0,(\psi_s)_{s\in[0,t]})=n_t.$$
\end{lemma}

{
\begin{lemma}\label{Lncont}
The map $t\mapsto n_t$ is locally constant.
\end{lemma}}

\begin{proof}[Proof of Lemma \ref{LDMIgraphselec} ]
We fix $t\in\R$. Let $\gamma: [0, 1]\rightarrow \psi_t(\cl_0)$ be a path such that for $i=0, 1$, $q_i=\pi(\gamma(i))\in U_t$ and $\gamma(i)=du_t(q_i)$. For $\tau\in[0, 1]$, we define a loop $\Gamma=\Gamma_\tau$ by
\begin{itemize}
\item $\forall s\in[0, 1], \Gamma_\tau(s)=T_{\psi_{t}^{-1}(\gamma(s\tau))}\cl_0$;
\item $\forall s\in[1, 2], \Gamma_\tau(s)=D\psi_{(s-1)t}\Big(T_{\psi_{t}^{-1}(\gamma(\tau))}\cl_0\Big)$;
\item $\forall s\in[2, 3], \Gamma_\tau(s)=T_{\gamma((3-s)\tau)}\psi_t(\cl_0)$;
\item $\forall s\in[3, 4], \Gamma_\tau(s)=D\psi_{(4-s)t}\Big(T_{\psi_{t}^{-1}(\gamma(0))}\cl_0\Big)$.
\end{itemize}
\begin{figure}
    \centering
    \includegraphics[scale=0.3]{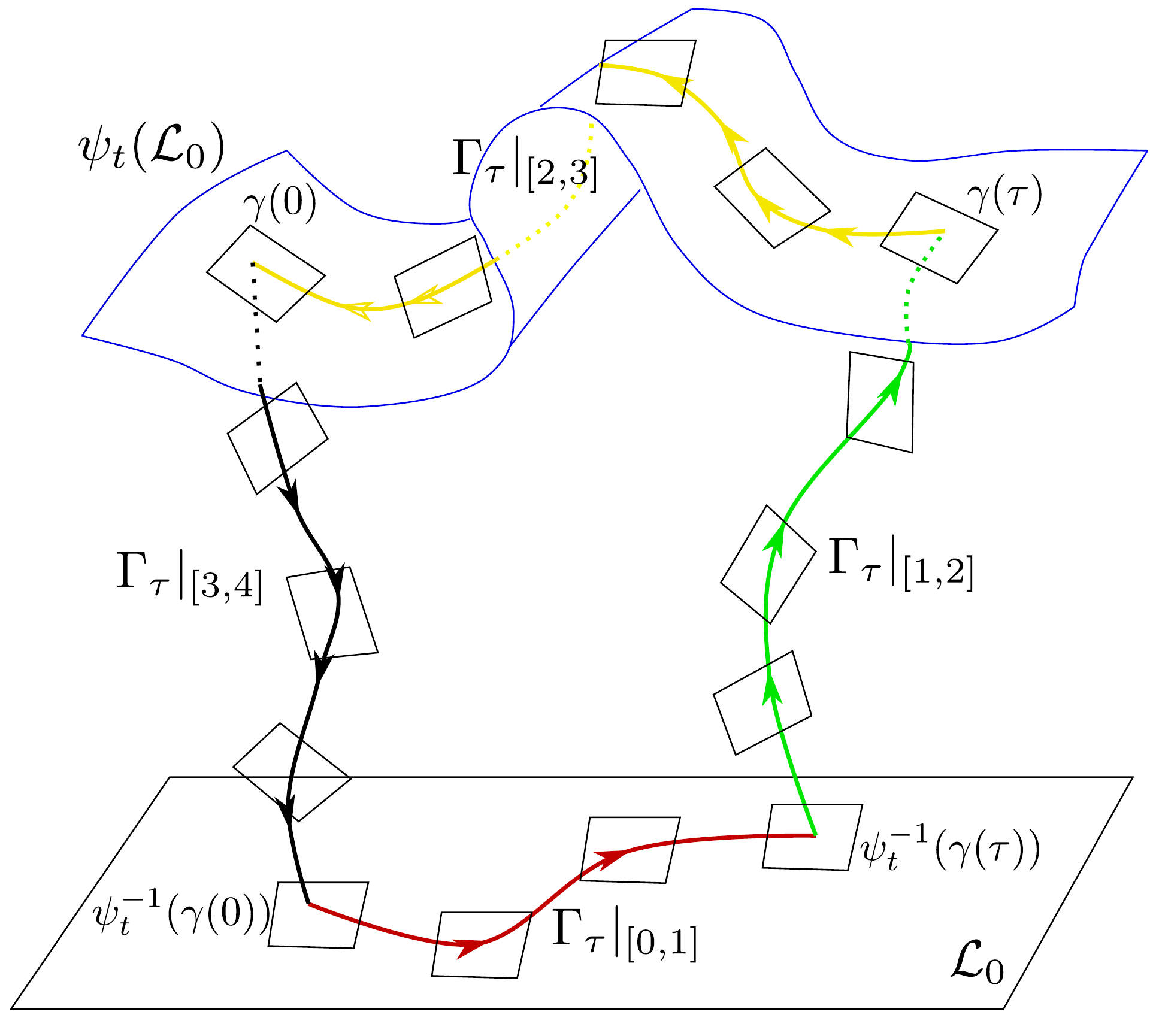}
    \caption{The loop $\Gamma_\tau$.}
    \label{fig:1}
\end{figure}
Along $\Gamma_{|[0, 1]}$, the Maslov index is zero because the path is on the zero section $\cl_0$. {Along $\Gamma_{|[1, 2]}$, the Maslov index is $\text{MI}\Big(\big( D\psi_{s}( T_{\psi_{t}^{-1}(\gamma(\tau))}\cl_0\big)_{s\in[0,t]}\Big)$. 
	Along $\Gamma_{|[2, 3]}$ the Maslov index is $-\text{MI}\Big(\big(  T_{\gamma(s)}\psi_t(\cl_0)\big) _{s\in[0,\tau]}\Big)$. Along $\Gamma_{|[3, 4]}$, the Maslov index is $-\text{MI}\Big(\big( D\psi_{s} (T_{\psi_{t}^{-1}(\gamma(0))}\cl_0)\big)_{s\in[0,t]}\Big)$.  Hence the total Maslov index along $\Gamma_\tau$ is 
\begin{equation*}
\text{MI}\Big(\big( D\psi_{s}( T_{\psi_{t}^{-1}(\gamma(\tau))}\cl_0\big)_{s\in[0,t]}\Big)-\text{MI}\Big(\big(  T_{\gamma(s)}\psi_t(\cl_0)\big) _{s\in[0,\tau]}\Big)-\text{MI}\Big(\big( D\psi_{s} (T_{\psi_{t}^{-1}(\gamma(0))}\cl_0)\big)_{s\in[0,t]}\Big).
\end{equation*}}
As $\tau\mapsto \Gamma_\tau$ is an homotopy, the total Maslov index along $\Gamma_\tau$ doesn't depend on $\tau$. Observe that
\begin{itemize}
\item for $\tau=0$, this index is $0$;
\item thanks to   Proposition \ref{Cselecindconst}, we have   $\text{MI}\Big(\big(  T_{\gamma(s)}\psi_t(\cl_0)\big) _{s\in[0,1]}\Big)=0$. Hence the total Maslov index along $\Gamma_1$ is 
$$
0=\text{MI}\Big(\big( D\psi_{s}( T_{\psi_{-t}(\gamma(1))}\cl_0\big)_{s\in[0,t]}\Big)-\text{MI}\Big(\big( D\psi_{s} (T_{\psi_{t}^{-1}(\gamma(0))}\cl_0)\big)_{s\in[0,t]}\Big).
$$
\end{itemize}
\end{proof}

\begin{proof}[Proof of Lemma \ref{Lncont}] 
Let us fix {$(q_0,t_0)\in \cU$}. {By Proposition \ref{tUouvert} and continuity of $(\psi_s)$, there exists $\varepsilon>0$ such that
$$\forall t\in (t_0-\varepsilon, t_0+\varepsilon), ({ \pi (} \psi_t\circ \psi_{t_0}^{-1}({ du_{t_0} (q_0)}){)},t)\in\cU.$$
We denote $\gamma(t)=\psi_t\circ \psi_{t_0}^{-1}({ du_{t_0} (q_0)})$ and $\cl_t=\psi_t(\cl_0)$. Then the arc 
$t\in(t_0-\varepsilon, t_0+\varepsilon)\mapsto T_{\gamma(t)}\cl_t $ doesn't intersect the singular cycle. Hence $t\in (t_0-\varepsilon, t_0+\varepsilon)\mapsto n_t$ is constant.

}

\end{proof}
{Observing that $n_0=0$, we combine the two lemmata for $t=1$ to get that ${\rm DMI}(T_x\cl_0,(\psi_s)_{s\in[0,1]})=0$ for all $x\in \psi_1^{-1}({\rm graph}(du_{1|U_1}))= \phi_1^{-1}({\rm graph}((\eta_1+du_1)_{|U_1}))$. 
Since $\psi$ is obtained by composing $\phi$ by a vertical translation (see paragraph \ref{exactification}), the Maslov index is the same for $\phi_t$ and $\psi_t$ (see Remark { at the end of section}  \ref{rmk invariance diffeo cs}), and Theorem \ref{Tppal} is proved in the case where $\cl$ is the zero section, taking $u=u_1$, $\eta=\eta_1$ and $U=U_1$.

Let us now assume only that $\cl$ is a Lagrangian graph, i.e., the graph of a closed $1$-form $\nu$.} We recall that all the diffeomorphisms $T_t:\M\righttoleftarrow$ defined by $T_t(p)=p+t\nu$ are symplectic.\\
Using $(T_t)$ and $(\phi_t)$, we will define an isotopy $(F_t)_{t\in [0, 1]}$ such that $F_0={\rm Id}_\M$ and $F_1(\cl_0)=\phi_1(\cl)$. Let $\alpha:[0, 1]\rightarrow [0,1]$ be a {smooth} non-decreasing function 
such that $\alpha(0)=0$, $\alpha(1)=1$ and $\alpha$ is constant equal to $\frac{1}{2}$ when restricted  to some neighbourhood of $\frac{1}{2}$. We the introduce $(F_t)$ by
\begin{itemize}
\item for $t\in [0, \frac{1}{2}]$, $F_t=T_{2\alpha({t})}$;
\item for $t\in [\frac{1}{2}, 1]$, $F_t=\phi_{2\alpha(t)-1}\circ T_1$.  
\end{itemize} 
The isotopy $(F_t)$  is an isotopy of conformally symplectic  diffeomorphisms such that $F_0= {\rm Id}_{{\cal M}}$. Applying the first case of this proof, there exist a closed 1-form $\eta$ of $M$ and a Lipschitz map $u:M\rightarrow \R$ that is {$C^r$} 
on an open subset $U$ of $M$ with full Lebesgue measure such that
$${\rm graph}(\eta+du)_{|U}\subset F_1(\cl_0)=\phi_1(\cl)$$
and
$$\forall x\in F_1^{-1}\big( {\rm graph}(\eta+du)_{|U}\big),\quad {\rm DMI}(T_x\cl_0, (F_t)_{t\in[0, 1]})=0.$$
Observe that the path $(DF_tT_x\cl_0)_{t\in[0, \frac{1}{2}]}=(DT_{2\alpha(t)})_{t\in[0, \frac{1}{2}]}$ has zero Maslov index since all these Lagrangian subspaces are 
transverse to the vertical.  Hence 
$${\rm DMI}(T_x\cl_0, (F_t)_{t\in[0, 1]})={\rm DMI}({T_{T_1(x)}\cl}, (F_t)_{t\in[\frac{1}{2}, 1]}).$$
The isotopy $(F_t)_{t\in[\frac{1}{2}, 1]}$ is just a reparametrization of the isotopy $(\phi_t)_{t\ in [0, 1]}$, hence we obtain finally
$$ \forall q\in U, p:=\phi^{-1}_1(\eta(q)+du(q))\in \cl\quad\text{and}\quad {\text{\rm DMI}\Big(T_{p}\cl,(\phi_s)_{s\in[0,1]}
	\Big)=0}.$$

\section{Angular Maslov index}\label{angular MI}
There are different approaches to Maslov index, at least three of these are contained in \cite{BargeGhys1992}. To prove some of our results, we will use the second approach that we explain now.
\subsection{Definition of the angular Maslov index}\label{def angular MI subsection}
{\color{black}In this section as in sub-section \ref{ssremindersMslov}, we assume that $(\M,\omega)$ is a $2d$-dimensional  symplectic manifold that admits a Lagrangian foliation $\cal V$. We denote by  $V(x)=V_x:=T_x\cal V$ its associated Lagrangian bundle. We endow $\M$   with an almost complex structure $J:T\M\righttoleftarrow$ that is compatible with $\omega$. We briefly recall that this means that
  \begin{itemize}
  \item for every $x\in\M$, $J_x:T_x\M\righttoleftarrow$ is linear and $J^2=-{\rm Id}_{T\M}$;
  \item $ $every $J_x$ is symplectic;
  \item for every $x\in\M$, the symmetric bilinear form $\omega (., J_x.)$ is positive definite. We denote $g:=\omega(., J.)$. 
  \end{itemize}
  The complex structure is then defined on every $T_x\M$ by 
  $$\forall (\lambda=\lambda_1+i\lambda_2, v)\in \C\times T_x\M, \lambda v=\lambda_1v+\lambda_2Jv.$$
The equality
  $$\forall x\in\M, \forall u, v\in T_x\M, \Theta(u, v)=g(u, v)+i\omega (u,v)$$
  define a positive definite Hermitian form on $T_x\M$. 
   We denote by $\cU(\M)$ the fiber bundle whose fibers $\cU_x(\M)$ are the unitary transformations of $T_x\M$.
  Observe that a real  $d$-dimensional linear subspace $L$ of the complex space $T_x\M$ is Lagrangian if and only if the Hermitian form $\Theta_x$ restricted to $L$ is real (and then $\Theta_x$ is  a real scalar product). Hence the group $\cU(\M)$  acts on the Lagrangian Grassmannian $\Lambda$. It is classical {\color{black}(see \cite[Lemma 3.10]{BatesWei97} or \cite{Aud03})} that the action of $\cU_x(\M)$ on $\Lambda_x$ is transitive.  \\
  If ${\rm Stab}_x$ is in the stabilizer of $V(x)$, then ${\rm Stab}_x$ preserves the scalar product that is the restriction of $\Theta_x$ to $V(x)$, i.e. is an orthonormal transformation   of $V(x)$.  Moreover, every orthogonal transformation of $V(x)$ can be extended to a unique unitary transformation of $T_xM$. We denote by  $\cO(\M)$ the fiber bundle whose fibers $\cO_x(\M)$ are these transformations that we call  orthogonal transformations of $T_x\M$. 
  There is a 
  {natural bijection} between $\cU_x(\M)/\cO_x(\M)$ and $\Lambda_x$ that maps  ${\rm Stab}_x$ on $V_x$. {This bijection sends each $[\phi]\in \cU_x(\M)/\cO_x(\M)$ to the Lagrangian space $\phi(V_x)\in\Lambda_x$, where $\phi\in\cU_x(\M)$ is a representative of $[\phi]$.} We denote by $\cR_x:\Lambda_x\rightarrow \cU_x(\M)/\cO_x(\M)$ its inverse bijection. 

The map $\delta_x:\cU_x(\M)\rightarrow \C^*$ defined by $\delta_x (\varphi)=\big({\rm det} \varphi\big)^2$ is a morphism of groups whose kernel contains $\cO_x(\M)$ and whose range is the set $U(1)$ of complex numbers with modulus 1. Hence we can define 
$\bar\delta_x:\cU_x(\M)/\cO_x(\M)\rightarrow U(1)$ and then $\Delta=\bar\delta\circ\cR:\Lambda\rightarrow U(1)$.
\begin{defi} Let $\Gamma:[a, b]\rightarrow \Lambda$ be a continuous map. Let $\theta:[a, b]\rightarrow \R$ be any continuous lift of $\Delta\circ \Gamma$, i.e. such that 
$$\forall t\in [a, b], \exp(i\theta(t))=\Delta(\Gamma(t)).$$
Then the {\sl angular Maslov index} of $\Gamma$ is 
\begin{equation}
\alpha\text{MI}(\Gamma):=\dfrac{\theta(b)-\theta(a)}{2\pi}.
\end{equation}
\end{defi}
  \begin{defi}
 Let $(\phi_t)
 $ be an isotopy of conformally symplectic diffeomorphisms of $\M$. Let $L\in\Lambda(\M)$ and $t>0$. Define the path
\[
s\in [0,t]\mapsto D\phi_sL\in\Lambda(\M).
\]
The  {\sl dynamical angular Maslov index} of $L$ at time $t$ is
\[
\text{D}\alpha\text{MI}(L,(\phi_s)_{s\in[0,t]})=\alpha\text{MI}((D\phi_sL)_{s\in[0, t]}).
\]
Whenever the limit exists, the {\sl asymptotic angular Maslov index} of $L$, is
\[
\text{D}\alpha\text{MI}_\infty(L,(\phi_s)_{s\in[0,+\infty]}):=\lim_{t\to+\infty}\dfrac{\text{D}\alpha\text{MI}(L,(\phi_s)_{s\in[0,t]})}{t}.
\]
 \end{defi}
 \begin{remk} In fact, {the existence of a Lagrangian foliation implies that} the bundle $\Lambda$ is trivial, diffeomorphic to $\M\times U(d)/O(d)$ where $U(d)$ and $O(d)$ are the groups of $d\times d$ unitary and orthogonal matrices respectively, see e.g. \cite[Section 1.2]{CGIP}. 
 \end{remk}
 
 {\begin{remk}
     Observe that the angular Maslov index is continuous with respect to the path $\Gamma$. Thus, the dynamical Maslov index at a fixed time $t$ is continuous with respect to $L\in\Lambda(\M)$, as long as the isotopy $(\phi_t)$ is at least $C^1$.
 \end{remk}}
 
 The following result is classical, see  \cite[Lemma 2.1]{CGIP}.
{\color{black}\begin{proposition}\label{indep lagrangian sub}
	Let $(\phi_t)
	$ be an isotopy of conformally symplectic diffeomorphisms of $\M$. Let $x\in \M$ and let $L_1,L_2\in\Lambda_x$. Then, for every $t>0$,
	\[
	\vert \text{D}\alpha\text{MI}(L_1,(\phi_s)_{s\in[0,t]})-\text{D}\alpha\text{MI}(L_2,(\phi_s)_{s\in[0,t]})\vert < 8d.
	\]
	In particular, whenever the asymptotic angular Maslov index at $x$ exists, it does not depend on the chosen Lagrangian subspace $L\in\Lambda_x$.
\end{proposition}}
  That is why we will often mention the asymptotic Maslov index at a point.
  \begin{proposition}\label{MI existence SCH}
  Let $(\phi_t)
  $ an isotopy of conformally symplectic diffeomorphisms of $\M$ such that $\phi_0={\rm Id}_\M$ and  $\phi_{t+1}= \phi_t\circ\phi_1$ (resp. $(\phi_t)_t$ is a flow). 
  If $\mu$ is a Borel probability measure with compact support that is invariant by $\phi_1$ (resp. by $(\phi_t)_t$), then the asymptotic Maslov index exists at $\mu$-almost every point $x\in\cm$. 
  \end{proposition} 
  \begin{proof} The proof uses methods of \cite[Section 4]{Schwartzman1957}. We assume that $\mu$ is ergodic: if not, using ergodic decomposition theorem, see e.g. \cite{Mane1987},  we deduce the result for $\mu$ from the result for ergodic measures. 
  
  Let us begin with  the  case when $(\phi_t)_t$ is a flow. The map $D\phi_t: \Lambda\rightarrow \Lambda$ defines a flow on $\Lambda$. Let $x\in\M$ be a regular point for $\mu$, i.e. such that the family of  measures $[x]_T$ defined by $[x]_T(f)=\frac{1}{T}\int_0^Tf(\phi_t(x))dt$ tends to $\int_\cm fd\mu$ for every continuous $f:\M\rightarrow \R$. Recall that $\mu$-almost every point $x\in\M$ is regular for $\mu$.  Let us fix $L_0\in \Lambda_x$ and let $\nu$ be any limit point at infinity of the family of measures $[L_0]_T$ defined by 
  $$\forall F\in C^0(\Lambda, \R), [L_0]_T(F)=\frac{1}{T}\int_0^TF(D\phi_tL_0)dt.$$ Then $\nu$ is an invariant measure for $(D\phi_t)$, see \cite{BobKry1957},  such that $p_*\nu=\mu$, where $p:T\M\to\M$ is the canonical projection.
 We have defined on $\Lambda$ the continuous function  $\Delta:\Lambda\rightarrow U(1)$. 
 A direct result of \cite[Section 4]{Schwartzman1957}, is that 
  $\text{D}\alpha\text{MI}_\infty(L,(\phi_s)_{s\in[0,+\infty]})$ exists and is finite at $\nu$-almost every point $(x, L)\in \Lambda$.  Since the asymptotic angular Maslov index of $L$ does not depend on the chosen Lagrangian subspace (see Proposition \ref{indep lagrangian sub}), we conclude that it exists at $p_*\nu=\mu$-almost every point $x\in\M$.
  
  When $\phi_{t+1}=\phi_t\circ \phi_1$, we define a flow $(F_t)$ on $\T\times \Lambda$ by $F_t(s, L)= (t+s, D\phi_tL)$. Then we apply Schwartzman's result to the function $(t, L)\mapsto \Delta(L)$, this gives the wanted result. 
  
  \end{proof}

\subsection{The angles of a Lagrangian subspaces} 
Let $(\M,\omega)$ be a $2d$-dimensional symplectic manifold that admits a Lagrangian foliation $\cv$. Let $J$ be an almost complex structure compatible with $\omega$. We introduce the notion of angles of a Lagrangian subspace $L\in\Lambda$ with respect to $J\cv$. For details, we refer to \cite{LPS}. 

\begin{nota}
	For every $x\in \M$, we denote by $JV(x)$ the image by the isomorphism $J_x$ of the Lagrangian subspace $V(x)=T_x\cv$.
\end{nota}
\begin{proposition}[Section 1.4 in \cite{LPS}]\label{prop JP}
{Let $(E^{2d},\omega)$ be a symplectic vector space, endowed with a complex structure compatible with $\omega$. Fix a Lagrangian subspace $H\subset E$. For every Lagrangian subspace $L\subset E$ there exists a unique unitary isomorphism of $E$ denoted by $\Phi_{H,L}
$ such that
	\begin{itemize}
		\item $\Phi_{H,L}(H)=L$ ;
	\item $\Phi_{H,L}$ is diagonalizable relatively to a unitary  
	basis {of $E$ whose vectors are} in $H$, with eigenvalues of the form $e^{i\theta_j}$, $j=1,\dots,d$, with $$\theta_j\in\left]-\frac \pi 2,\frac \pi 2\right]\text{ for }j=1,\dots,d\, .$$
		\end{itemize}}

\end{proposition}

{In the sequel, we apply Proposition \ref{prop JP} to each symplectic vector space $(T_x\M,\omega_x)$, endowed with the almost complex structure $J$. The fixed Lagrangian subspace $H$ in each $T_x\M$ is $JV(x)$.}

\begin{defi}
	Let $L\in\Lambda_x$. The angles of $L$ with respect to $JV(x)$ is the equivalence class\[
	(\theta_1^{JV(x),L},\dots,\theta_d^{JV(x),L} )/\sim\, ,
	\]
	where \begin{itemize}
		\item $(\theta_1^{JV(x),L},\dots,\theta_d^{JV(x),L} )\in ]-\frac{\pi}{2},\frac{\pi}{2}]^d$ is the $d$-uplet composed by arguments of the $d$ eigenvalues given by Proposition \ref{prop JP} { applied at $T_x\M$ with respect to $JV(x)$ and $L$};
		\item $\sim$ is the equivalence relation obtained from permutations over the $d$-entries. 
	\end{itemize} 
\end{defi}

{Let us denote by}
$$\{ e_1(x),\dots,e_d(x) \}\, $$ 
{ a unitary basis of $T_x\cal M$ that is contained in $JV(x)$ and given by Proposition \ref{prop JP}.}
{We have} $$\C e_1(x)\oplus\dots\oplus\C e_d(x)=T_x\M\, ,$$
where $T_x\M$ is seen as a complex vector space. 
In particular, $$\{ e_1(x),\dots,e_d(x),J_xe_1(x),\dots,J_xe_d(x) \}$$ is a symplectic basis of $T_x\M$, seen as a real vector space of dimension $2d$. 

{Observe that for every $x\in\M$ and every $v\in T_x\M$ {it holds $J_xv=iv=e^{i\frac{\pi}{2}}v$}. Refering then to notations introduced in Subsection \ref{def angular MI subsection}, the image $\cR_x(JV(x))$ is the equivalence class of {$J_x\in \cU_x(\M)$}. 
Thus, for $L\in\Lambda_x$ we have that $\cR_x(L)$ is the equivalence class of the unitary transformation $\Phi_{JV(x),L}\circ J_x$. Consequently, since $\Delta_x=\bar\delta_x\circ\cR_x$, it holds
\begin{equation}\label{Delta and angles}
\Delta_x(L)=
\left(\mathrm{det}\left(\Phi_{JV(x),L}\circ J_x\right)\right)^2= \exp\Big(2i\sum_{j=1}^d\theta_j^{JV(x),L}\Big)\exp(i\,d\pi)\, .
\end{equation}}

Let $\Gamma:[a,b]\to \Lambda$ be a continuous map. 
\begin{nota}
To ease the notation, for every $t\in[a,b]$, we denote the angles of $\Gamma(t)$ (with respect to $JV(p\circ\Gamma(t))$) as
\[
\left( \theta_1^{JV,\Gamma}(t),\dots,\theta_d^{JV,\Gamma}(t) \right)/\sim\, .
\]
\end{nota}
The angular Maslov index $\alpha\mathrm{MI}(\Gamma)$ differs by an integer from the angular quantity
\begin{equation}\label{relation AMI and angles}
\dfrac{1}{\pi} \left(\sum_{j=1}^d\left( \theta_j^{JV,\Gamma}(b)-\theta_j^{JV,\Gamma}(a)\right)\right)\, ,
\end{equation}
{
since the angular Maslov index is a continuous lift of the function $\Delta$ and because of Equation \ref{Delta and angles}. We will see in the next paragraph that this integer is actually $\mathrm{MI}(\Gamma)$.}
~\newline


%
%

\begin{remk}
	We have that $\mathrm{dim}(L\cap V(x))=k$, for some $0\leq k\leq n$, if and only if exactly $k$ angles of $L$ with respect to $JV(x)$ are equal to $\frac \pi 2$.
\end{remk}

\subsection{Relation between Maslov index and angular Maslov index}

The following proposition clarifies the relation between Maslov index and angular Maslov index.

\begin{proposition}\label{lemma relation MI et AMI}
	Let $\Gamma:[a,b]\to \Lambda$ be a smooth path such that
	$$
	\Gamma(a)\cap V(p\circ\Gamma(a))=\Gamma(b)\cap V(p\circ\Gamma(b))=\{0\}\, .
	$$
	 Then
	 \begin{equation}\label{eq of prop ami mi}
	 \alpha\mathrm{MI}(\Gamma)=\frac{1}{\pi}\left( \sum_{j=1}^d(\theta_j^{JV,\Gamma}(b)-\theta_j^{JV,\Gamma}(a)) \right)+\mathrm{MI}(\Gamma)\, .
	 \end{equation}
\end{proposition}

\begin{proof}{
Without loss of generality, assume that the path $\Gamma$ is in general position with respect to $\Sigma(\M)=\{L\in\Lambda(\M) :\ L\cap V(p(L))\neq \{0\}\}$. Let $t\in]a,b[$ be a crossing. Since $\Gamma$ is in general position, $\Gamma(t)$ has exactly only one angle equal to $\frac{\pi}{2}$ with respect to $JV$. {Up to a permutation over angles, we can assume that $\theta_1^{JV,\Gamma}(t)=\frac{\pi}{2}$.}

Let $\epsilon>0$ be small enough such that
\begin{itemize}
\item for $s\in[t-\epsilon,t+\epsilon]\setminus\{t\}$ it holds $\Gamma(s)\cap V(p\circ\Gamma(s))=\{0\}$;
\item for $s\in[t-\epsilon,t+\epsilon]$ it holds that, for all $j>1$, $$|\theta^{JV,\Gamma}_1(s)|>|\theta_j^{JV,\Gamma}(s)|\, .$$ 
\end{itemize}
It will be sufficient to show that Equation \ref{eq of prop ami mi} holds for the subpath $\Gamma_{|[t-\epsilon,t+\epsilon]}$.

Let us start by calculating the angular Maslov index of $\Gamma_{|[t-\epsilon,t+\epsilon]}$:
	\begin{equation*}
	\alpha\text{MI}(\Gamma_{|[ t-\epsilon,t+\epsilon]})=
	\end{equation*}
	\begin{equation*}
	\dfrac{\theta_1^{JV,\Gamma}(t+\epsilon)-\theta_1^{JV,\Gamma}(t-\epsilon)}{\pi}+\dfrac{1}{\pi}\Big(\sum_{j=2}^d\theta_j^{JV,\Gamma}(t+\epsilon)-\theta_j^{JV,\Gamma}(t-\epsilon\Big)+k,
	\end{equation*}
	where \[
	k=\begin{cases}
	+1\quad\text{if }-\dfrac{\pi}{2}<\theta_1^{JV,\Gamma}( t+\epsilon)<0<\theta_1^{JV,\Gamma}( t-\epsilon)<\dfrac{\pi}{2},\\
	\\
	-1\quad\text{if }-\dfrac{\pi}{2}<\theta_1^{JV,\Gamma}( t-\epsilon)<0<\theta_1^{JV,\Gamma}( t+\epsilon)<\dfrac{\pi}{2}.
	\end{cases}
	\]}
	
{
 	Let us now calculate $\mathrm{MI}(\Gamma_{|[t-\epsilon,t+\epsilon]})$. We can smoothly perturb the path $\Gamma|_{[t-\epsilon,t+\epsilon]}$ into a Lagrangian path $\tilde\Gamma:{[ t-\varepsilon,t+\varepsilon] }\to \Lambda(\M)$  
such that
	\begin{itemize}
	\item[$(i)$] $\mathrm{MI}(\Gamma_{|[t-\epsilon, t+\epsilon]})=\mathrm{MI}(\tilde\Gamma)$;\vspace{4pt}
		\item[$(ii)$] $\tilde\Gamma$ is in general position with respect to $\Sigma$, $\tilde\Gamma$ has a crossing at $0$ with $\Sigma$ and $\tilde\Gamma(s)\cap V(p\circ\tilde\Gamma(s))=\{0\}$ for ${ s\in[ t-\varepsilon,t+\varepsilon]\backslash \{ t\} }$ 
		; \vspace{4pt}
		
		\item[$(iii)$] $\tilde\Gamma$ is in general position with respect to $\{L\in\Lambda(\M) :\ L\cap JV(p(L))\neq \{0\}\}$ and { $\tilde\Gamma(t)\cap JV(p\circ\tilde\Gamma(t))=\{0\}$} 
		.
		
		
		
	\end{itemize}
	
Conditions $(i)$ and $(ii)$ can be obtained easily, see Section \ref{section def MI}, and they are stable under small perturbations. Moreover, since being in general position is a dense and open condition, we can assume, up to perturb $\Gamma$, that the initial path is also in general position with respect to $\{L\in\Lambda(\M) :\ L\cap JV(p(L))\neq \{0\}\}$. 

To obtain $\tilde\Gamma$, we need to perturb $\Gamma_{|[t-\epsilon,t+\epsilon]}$ so that the new path $\tilde\Gamma$ does not intersect the horizontal $JV$ at { time $t$} .

Two cases can happen.\vspace{4pt}
	
	\noindent $(1)$ $\Gamma( t)\cap JV(p\circ\Gamma(t))=\{0\}$. Then we conclude by defining { $\tilde\Gamma=\Gamma$} 
	.\vspace{4pt}

	\noindent $(2)$ $\Gamma(t)\cap JV(p\circ\Gamma(t))\neq \{0\}$. In this case, because of the general position assumption, the subspace $\Gamma(t)\cap JV(p\circ\Gamma(t))$ is 1-dimensional, generated by one vector $w$. Let $0<\theta \ll 1$, complete $w$ to a {unitary} basis and consider the unitary transformation $R$ that rotates by $e^{i\theta}$ the vector $w$ and that is the identity on the other vectors {of the basis}. Up to select $\theta$ small enough, the Lagrangian path {$\tilde\Gamma := R\circ\Gamma $  } 
	is a small perturbation of $\Gamma_{|[t-\epsilon,t+\epsilon]}$. Up to select a subpath of $\tilde\Gamma$, the defined path satisfies all the required conditions. 
	

To calculate the Maslov index $\mathrm{MI}(\Gamma|_{[t-\epsilon,t+\epsilon]})$, we calculate then $\mathrm{MI}(\tilde\Gamma)$, because of condition $(i)$.
In particular, up to select a subpath, we can assume that
\[
\tilde\Gamma:{ [t-\epsilon,t+\epsilon]} \to\Lambda(\M)
\]
is in general position with respect to $\Sigma$, it has a unique crossing with the vertical at {$s=t$}, for all $j>1$ and all {
$s\in[t-\epsilon,t+\epsilon]$} 
it holds $$|\theta_1^{JV,\tilde\Gamma}(s)|>|\theta_j^{JV,\tilde\Gamma}(s)|$$ and $\tilde\Gamma(s)\cap JV(p\circ\tilde\Gamma(s))=\{0\}$ for all {
$s\in[t-\epsilon,t+\epsilon]$} 
.

For every {
$s\in[t-\epsilon,t+\epsilon]$} 
, by Proposition \ref{prop JP}, we have a unitary basis { of $T_{p(\tilde\Gamma(s))}\M$ whose vectors are} in $JV(p\circ\tilde\Gamma(s))$
$$
\{v_1(s),v_2(s),\dots,v_d(s)\}
$$
made up of eigenvectors relative to the eigenvalues $e^{i\theta_j^{JV,\tilde\Gamma}(s)}$, $j=1,\dots,d$ {such that $(e^{i\theta_j^{JV,\tilde\Gamma}(s)}v_j)$, which is also a unitary basis of $T_{p(\tilde\Gamma(s))}\M$, is a basis of $\tilde \Gamma(s)$ over $\R$}. 

We want then to consider the variation of the index of the quadratic form $$Q_{JV(p\circ\tilde\Gamma(s)}(V(p\circ\tilde\Gamma(s)),\tilde\Gamma(s))\, .$$
Up to a sign change, we can work with the quadratic form
$$
{Q=} Q_{JV(p\circ\tilde\Gamma(s)}(\tilde\Gamma(s),V(p\circ\tilde\Gamma(s))\, .
$$
In the sequel, we denote by $Q$ both the quadratic form and the associated bilinear form.
We consider the basis $${(E_j)_{1\leq j\leq d} =}(P^{JV(p\circ\tilde\Gamma(s))}(e^{i\theta_j^{JV,\tilde\Gamma}(s)}v_j(s)))_{1\leq j\leq d}$$
of $T_{p\circ\tilde\Gamma(s)}\M/JV(p\circ\tilde\Gamma(s))$. {
Then we have for all $j, k$
\[\begin{split}Q\big(E_j&, E_k\big)\\
&=\frac{1}{2}\Big( \omega( e^{i\theta_j^{JV,\tilde\Gamma}(s)}v_j(s), i\sin(\theta_k^{JV,\tilde\Gamma}(s))v_k)+ \omega( e^{i\theta_k^{JV,\tilde\Gamma}(s)}v_k(s), i\sin(\theta_j^{JV,\tilde\Gamma}(s))v_j) \Big)
\end{split}
\]
We deduce that $(E_j)_{1\leq j\leq d}$ is orthogonal for $Q$ and that for all $j\in\{ 1,\dots, d\}$ 
\[\begin{split}
{Q(E_j,E_j)}&=\omega( e^{i\theta_j^{JV,\tilde\Gamma}(s)}v_j(s), i\sin(\theta_j^{JV,\tilde\Gamma}(s))v_j)\\
&=\omega (\cos(\theta_j^{JV,\tilde\Gamma}(s))v_j(s)+\sin(\theta_j^{JV,\tilde\Gamma}(s)Jv_j(s),\sin(\theta_j^{JV,\tilde\Gamma}(s))Jv_j(s)))\\
&= \cos(\theta_j^{JV,\tilde\Gamma}(s))\sin(\theta_j^{JV,\tilde\Gamma}(s))=\frac{1}{2}\sin\big(  2\theta_j^{JV,\tilde\Gamma}(s)    \big)
\end{split}\]}





We can thus conclude that
	\begin{equation*}
	\mathrm{MI}(\Gamma|_{[t-\epsilon,t+\epsilon]})=\mathrm{MI}(\tilde\Gamma)=\begin{cases}
	+1\quad\text{if }-\frac \pi 2 <\theta_1^{JV,\tilde\Gamma}({ t+\varepsilon})<0<\theta_1^{JV,\tilde\Gamma}( { t-\varepsilon})<\frac\pi 2\, ,\\
	\\
	-1\quad\text{if }-\frac \pi 2 <\theta_1^{JV,\tilde\Gamma}( { t-\varepsilon})<0<\theta_1^{JV,\tilde\Gamma}({ t+\varepsilon})<\frac\pi 2\, .
	\end{cases}
	\end{equation*}

}\end{proof}

From Proposition \ref{lemma relation MI et AMI} we immediately obtain the following results.
\begin{cor}\label{comparison DAMI et DMI}
	Let $\M$ be a $2d$-dimensional symplectic manifold that admits a Lagrangian foliation. Let $(\phi_t)
	$ be an isotopy of conformally symplectic diffeomorphisms of $\M$. For every $L\in\Lambda(\M)$ and $t>0$ it holds
	\begin{equation}
	\vert\text{D}\alpha\text{MI}(L,(\phi_s)_{s\in[0,t]})-\text{DMI}(L,(\phi_s)_{s\in[0,t]})\vert <  d\, .
	\end{equation}
	In particular, whenever the asymptotic angular Maslov index exists at $x\in\M$, it does not depend on the chosen Lagrangian subspace and it holds
	\[
	\mathrm{D}\alpha\mathrm{MI}_{\infty}(x,(\phi_t))=\mathrm{DMI}_\infty(x,(\phi_t))\, .
	\]
\end{cor} 

\subsection{Independence of the aymptotic Maslov index from the isotopy}

The index $\mathrm{D}\alpha\mathrm{MI}$ does not depend on the chosen conformally symplectic isotopy.

\begin{proposition}\label{independence isotopy}
	Let $\phi$ be a conformally symplectic diffeomorphism isotopic to the identity on $\M$. Let $(\phi_t)_{t\in[0,1]}, (\psi_t)_{t\in[0,1]}$ be isotopies of conformally symplectic diffeomorphisms such that $\phi_0=\psi_0=\mathrm{Id}_{T^*M}$ and $\phi_1=\psi_1=\phi$.
	Then for every $L\in\Lambda$
	$$
	\mathrm{D}\alpha\mathrm{MI}(L, (\phi_t)_{t\in[0,1]})=\mathrm{D}\alpha\mathrm{MI}(L, (\psi_t)_{t\in[0,1]})\, .
	$$
	Extend then each isotopy on $[0,+\infty)$ by asking that $\phi_{1+t}=\phi_t\circ\phi$ and $\psi_{1+t}=\psi_t\circ\phi$. Thus,
 whenever the limit exists, the asymptotic angular Maslov index does not depend on the chosen isotopy, i.e. 
	$$
	\mathrm{D}\alpha\mathrm{MI}_\infty(p(L),\phi):=\mathrm{D}\alpha\mathrm{MI}_\infty(p(L), (\phi_t)
	)=\mathrm{D}\alpha\mathrm{MI}_\infty(p(L), (\psi_t)
	)\,.
	$$
	
\end{proposition}
\begin{proof}
Since $\phi_1=\psi_1=\phi$ and from \eqref{relation AMI and angles}, for every $L\in\Lambda$ it holds
	\[
	\text{D}\alpha\text{MI}(L,(\phi_t)_{t\in[0,1]})=\text{D}\alpha\text{MI}(L,(\psi_t)_{t\in[0,1]})+ 2k_L,
	\]
	for some $k_L\in\Z$. The function
	$$ L\mapsto \mathrm{D}\alpha\mathrm{MI}(L,(\phi_t)_{t\in[0,1]})-\mathrm{D}\alpha\mathrm{MI}(L,(\psi_t)_{t\in[0,1]})
	$$
	is continuous. Therefore, the constant $k=k_L\in\Z$ does not depend on $L\in\Lambda$. To conclude, it is sufficient to find $L\in\Lambda$ such that \begin{equation}\label{goal eq}
	\text{D}\alpha\text{MI}(L,(\phi_t)_{t\in[0,1]})=\text{D}\alpha\text{MI}(L,(\psi_t)_{t\in[0,1]})	\,.\end{equation}
	Consider then a Lagrangian graph $\cl\subset\M$. By Theorem \ref{Tppal}, with $\eta,U$ and $u$ defined as in the statement of Theorem \ref{Tppal}, for every $x\in\phi^{-1}(\text{graph}((\eta+du)_{|U}))$ it holds
	\begin{equation}\label{DMI on graph sel}
	\mathrm{DMI}(T_x\cl,(\phi_t)_{t\in[0,1]})=\mathrm{DMI}(T_x\cl,(\psi_t)_{t\in[0,1]})=0.
	\end{equation}
	Let then $\bar x$ be a point in $ \phi^{-1}(\text{graph}((\eta+du)_{|U}))\subset \cl$. From Proposition \ref{lemma relation MI et AMI} and from \eqref{DMI on graph sel}, it holds
	\[
	\text{D}\alpha\text{MI}(T_{\bar x}\cl,(\phi_t)_{t\in[0,1]})-\text{D}\alpha\text{MI}(T_{\bar x}\cl,(\psi_t)_{t\in[0,1]})=\]\[\dfrac{1}{\pi}\Big(\sum_{j=1}^d\theta_j^{JV,D\phi_t(T_{\bar x}\cl)}(1)-\theta_j^{JV,D\phi_t(T_{\bar x}\cl)}(0)\Big)-\dfrac{1}{\pi}\Big(\sum_{j=1}^d\theta_j^{JV,D\psi_t(T_{\bar x}\cl)}(1)-\theta_j^{JV,D\psi_t(T_{\bar x}\cl)}(0)\Big)\, .
	\]
	Since $D\phi_1(T_{\bar x}\cl)=D\psi_1(T_{\bar x}\cl)=D\phi(T_{\bar x}\cl)$, the second term of the last equality is zero, as required.

\end{proof}
{
We may now deduce the
\begin{proof}[Proof of Proposition \ref{Indisotopybis}]
Since the difference between the angular Maslov index and the Maslov index in Proposition \ref{lemma relation MI et AMI} only depends on $\Gamma(b)$ and $\Gamma(a)$, the results of Proposition \ref{independence isotopy} also hold for Maslov index.\end{proof}}
From Corollary \ref{comparison DAMI et DMI} and Proposition \ref{indep lagrangian sub}, we deduce the following result.
\begin{cor}\label{no subspace}
	Let $(\phi_t)
	$ be an isotopy of conformally symplectic diffeomorphisms of $\M$. For every $x\in \M$ the asymptotic Maslov index, whenever it exists, does not depend on the chosen Lagrangian subspace $L\in\Lambda_x$.
\end{cor}

Moreover, the following holds.
\begin{cor}
Let $(\phi_{1,t})_t, (\phi_{2,t})_t$ be two isotopies of conformally symplectic diffeomorphisms of $\M$ such that $\phi_{1,0}=\phi_{2,0}=\mathrm{Id}_{\M}, \phi_{1,1}=\phi_{2,1}$ and $\phi_{i,1+t}=\phi_{i,t}\circ\phi_{i,1}$ for $i=1,2$.
	Then for every $x\in\M$, whenever the limit exists,
	$$
	\text{DMI}_\infty(x,(\phi_{1,t}))=	\text{DMI}_\infty(x,(\phi_{2,t}))\, .
	$$
\end{cor}

%
%

 \section{Applications and proofs of main outcomes}\label{section applications}

This section is devoted to the proofs of the main consequences presented in the introduction and further interesting applications.
%
%
%
\subsection{Proof of Corollary \ref{cormeasurebis}}

Let $(\phi_t)_{t\in\R}$ be a conformally symplectic isotopy of ${\cal M}$ such that $\phi_0=\mathrm{Id}_{\cal M}$ and $\phi_{t+1}=\phi_t\circ\phi_1$.
Let $\cl\subset {\cal M} $ be a Lagrangian submanifold   that is Hamiltonianly isotopic to a graph and such that $\displaystyle{\bigcup_{t\in[0, +\infty)}\phi_t(\cl)}$ is relatively compact.\\
More precisely, let $\cl_0\subset\M$ be a Lagrangian graph and let $(h_t)_{t\in[0,1]}$ be a Hamiltonian isotopy such that $h_0=\mathrm{Id}_{\M}$ and $h_1(\cl_0)=\cl$. Let $\alpha:[0,1]\to[0,1]$ be a smooth non-decreasing function such that $\alpha(0)=0$ and $\alpha$ is constant equal to $1$ when restricted to some neighborhood of $1$. {Let $\beta:[0,1]\to[0,1]$ be a smooth non-decreasing function such that $\beta$ is constant equal to $0$ on some neighborhood of $0$ and equal to the identity on some neighborhood of $1$.} Define then $(\psi_t)_{t\in[0,+\infty)}$ as
\[
{\psi_t:=\begin{cases}
h_{\alpha(t)}\quad \text{for }t\in\left[0,1\right]\, , \\
\\
\phi_{\beta(t-1)}\circ h_1\quad\text{for }t\in[1,2]\, , \\
\\
\phi_{t-2}\circ h_1\quad \text{for }t\in[2,+\infty)\, .
\end{cases}}
\]
Then $(\psi_t)
$ is an isotopy of conformally symplectic diffeomorphisms such that $\psi_0=\mathrm{Id}_\M$ { and $\psi_t(\cl_0)=\phi_{t-2}(\cl)$ for $t\geq 2$}.

Applying then Theorem \ref{Tppal} to the Lagrangian graph $\cl_0$ with respect to the isotopy $(\psi_t)
$, for every $t\in[1,+\infty)$ there exists at least a point $x_t\in\cl_0$ such that
\begin{equation}\label{thm 1.1 a L0}
\mathrm{DMI}(T_{x_t}\cl_0, (\psi_s)_{s\in[0,t]})=0\, .
\end{equation}
By compactness of $\cl_0$, by the relation between $\mathrm{DMI}$ and $\mathrm{D}\alpha\mathrm{MI}$ (see  { Corollary} \ref{comparison DAMI et DMI}) and by the continuity of the angular Maslov index, there exists a constant $C>0$ such that for every $x\in\cl_0$ it holds
\begin{equation}\label{control over hs}
\vert \mathrm{DMI}(T_x\cl_0, (\psi_t)_{t\in[0,1]})\vert \leq C\, .
\end{equation}
From \eqref{thm 1.1 a L0} and \eqref{control over hs}, for every $t\in[0,+\infty)$ we have then a point $z_t:= h_1(x_{t+1})\in h_1(\cl_0)=\cl$ such that
\begin{equation}\label{control over phis}
\vert \mathrm{DMI}(T_{z_t}\cl, (\phi_s)_{s\in[0,t]})\vert \leq C\, .
\end{equation}
Consider then the sequence $(z_n)_{n\in\N}$ in $\cl$. For every $n\in\N$, we define the following probability measure on $\Lambda(\M)$:
\[
\mu_n := \dfrac{1}{n}\sum_{i=0}^{n-1} \delta_{D\phi_i(T_{z_n}\cl)}\, ,
\]
where $\delta_*$ is the Dirac measure supported on $*\in\Lambda(\M)$.  Since $\bigcup_{t\in[0,+\infty)}\phi_t(\cl)$ is relatively compact, we can extract a subsequence $(\mu_{n_k})_{k\in\N}$ that converges to a probability measure $\bar\mu$ on $\Lambda(\M)$. The measure $\bar\mu$ is $D\phi_1$-invariant. The projected measure 
{$\mu:=p_*\bar\mu$} is a $\phi_1$-invariant probability measure on $\M$.

By Corollary \ref{comparison DAMI et DMI} it holds
\[
\mathrm{DMI}(\mu,(\phi_t)
)=
\]
\[
\int_{\M}\mathrm{DMI}_\infty(x, (\phi_t)
)\, d\mu(x) = \int_\M \mathrm{D}\alpha\mathrm{MI}_\infty(x, (\phi_t)
)\, d\mu(x) \, .
\]
Since the asymptotic angular Maslov index does not depend on the chosen Lagrangian subspace, we have that
\[
\int_\M \mathrm{D}\alpha\mathrm{MI}_\infty(x, (\phi_t)
)\, d\mu(x)=\int_{\Lambda(\M)} \mathrm{D}\alpha\mathrm{MI}_\infty(p(L), (\phi_t)
)\, d\bar\mu(L)\, .
\]
Birkhoff's Ergodic Theorem, applied at the function $L\mapsto \mathrm{D}\alpha\mathrm{MI}(L,(\phi_t)_{t\in[0,1]})$ and at the probability measure $\bar\mu$ on $\Lambda(\M)$, assures us that

\[
\int_{\Lambda(\M)} \mathrm{D}\alpha\mathrm{MI}_\infty(p(L), (\phi_t)
)\, d\bar\mu(L)=\int_{\Lambda(\M)}\mathrm{D}\alpha\mathrm{MI}(L,(\phi_t)_{t\in[0,1]})\, d\bar\mu (L)\, .
\]
Since $(\mu_{n_k})_{k\in\N}$ converges to $\bar\mu$, it holds
\[\begin{split}
\int_{\Lambda(\M)}\mathrm{D}\alpha\mathrm{MI}(L,(\phi_t)_{t\in[0,1]})\, d\bar\mu (L)&=\lim_{k\to+\infty}\dfrac{1}{n_k}\sum_{i=0}^{n_k-1}\mathrm{D}\alpha\mathrm{MI}(D\phi_i(T_{z_{n_k}}\cl),(\phi_t)_{t\in[0,1]})\\
&=
\lim_{k\to+\infty}\dfrac{1}{n_k}\mathrm{D}\alpha\mathrm{MI}(T_{z_{n_k}}\cl,(\phi_t)_{t\in[0,n_k]})\, .
\end{split}
\]
From \eqref{control over phis} and from Corollary \ref{comparison DAMI et DMI}, we have that for every $k\in\N$
\[\vert \mathrm{D}\alpha\mathrm{MI}(T_{z_{n_k}}\cl,(\phi_t)_{t\in[0,n_k]})\vert \leq C+d\, .\] Thus, we conclude that
\[
\mathrm{DMI}(\mu,(\phi_t)
)=\lim_{k\to+\infty}\dfrac{1}{n_k}\mathrm{D}\alpha\mathrm{MI}(T_{z_{n_k}}\cl,(\phi_t)_{t\in[0,n_k]})=0\, ,
\]
as required. Observe that the support of the measure $\mu$ is contained in 
{$$\displaystyle{\bigcap_{T\in [0, +\infty)}\overline{\bigcup_{t\in[T, +\infty)}\phi_t(\{ z_n :\ n\in\N \})}}\subset \displaystyle{\bigcap_{T\in [0, +\infty)}\overline{\bigcup_{t\in[T, +\infty)}\phi_t(\cl)}}\, .$$}

Let now $(\phi_t)
$ be a conformally symplectic flow on $\M$. We can then consider for every $t\in[0,+\infty)$ the measure on $\Lambda(\M)$
\begin{equation}\label{def mu_t}
\mu_t:=[z_t]_t=\frac 1 t \int_0^t\delta_{D\phi_s(T_{z_t}\cl)}\, ds\, .
\end{equation}
Observe, as before, that, from the choice of $z_t$, for every $t$ it holds
\begin{equation}\label{property z_t}
\vert \mathrm{D}\alpha\mathrm{MI}(T_{z_t}\cl,(\phi_s)_{s\in[0,t]})\vert \leq C+d\, .
\end{equation}
Consider then an accumulation point $\bar \mu$ of $(\mu_t)_{t\in[0,+\infty)}$ in the space of measure on $\Lambda(\M)$, which exists because $\bigcup_{t\in[0,+\infty)}\phi_t(\cl)$ is relatively compact. More precisely, let $(t_n)_{n\in\N}$ be a sequence such that $t_n\to+\infty$ and $\mu_{t_n}\rightharpoonup\bar\mu$ as $n\to+\infty$.
The measure $\bar\mu$ is $(D\phi_t)$-invariant. The projection {$\mu=p_*\bar\mu$} 
is then a $\phi_t$-invariant measure on $\M$. 

{We denote by $F$ the derivative of the function $\Delta$ that we introduced in section \ref{angular MI} in the direction of the vectofield $\chi$, where $\chi$ is the vectorfield associated to the flow $(D\phi_s):\Lambda({\cal M})\righttoleftarrow$. Then,}
for every $L\in\Lambda(\M)$ and every $t$ it holds}
\[
\mathrm{D}\alpha\mathrm{MI}(L,(\phi_s)_{s\in[0,t]})=\int_0^t F\circ D\phi_s(L)\, ds\, .
\]
{By Birkhoff Ergodic Theorem for flows (see \cite[Page 459]{NemSte60}), the following integral exists $\overline\mu$ almost everywhere 

}
\[
\bar F(L):=\lim_{t\to+\infty}\dfrac 1 t \int_0^t F\circ D\phi_s(L)\, ds=\mathrm{D}\alpha\mathrm{MI}_\infty(p(L),(\phi_t)
)\, ,
\]
{ and we have}
\[
\int_{\Lambda(\M)}\bar F(L)\, d\bar \mu(L) =\int_{\Lambda(\M)}F(L)\, d\bar \mu(L)\, .
\]
Following then the same calculus as for the previous case, it holds
\[\begin{split}
\mathrm{DMI}(\mu,(\phi_t)
)&= \int_\M \mathrm{DMI}_\infty (x,(\phi_t)
)\, d\mu(x) \\
&
=\int_\M\mathrm{D}\alpha\mathrm{MI}_\infty(x, (\phi_t)
)\, d\mu(x)\\
&=\int_{\Lambda(\M)}\bar F(L)\, d\bar\mu(L)=\int_{\Lambda(\M)}F(L)\, d\bar\mu(L)\, .
\end{split}\]
Since $\bar\mu=\lim_{n\to\infty}\mu_{t_n}$, because of \eqref{def mu_t} and from \eqref{property z_t}, we have that
\[\begin{split}
\int_{\Lambda(\M)}F(L)\, d\bar\mu(L)&=\lim_{n\to+\infty}\dfrac{1}{t_n}\int_0^{t_n}F\circ D\phi_s(T_{z_{t_n}}\cl)\, ds\\
&= \lim_{n\to+\infty}\dfrac{\mathrm{D}\alpha\mathrm{MI}(t_{z_{t_n}}\cl, (\phi_s)_{s\in[0,t_n]})}{t_n}=0\, .
\end{split}
\]
Thus, we conclude that $\mathrm{DMI}(\mu,(\phi_t)
)=0$, as desired.

%

\subsection{Proof of Corollary \ref{measuretorus}}

Let $(\phi_t)
$ be a   symplectic isotopy of $\T^{2d}$ such that $\phi_0=\mathrm{Id}_{\T^{2d}}$ and $\phi_{t+1}=\phi_t\circ\phi_1$. Using a covering $\Pi:T^*\T^d\to\T^{2d}$, we can lift the symplectic isotopy $(\phi_t)
$ on $\T^{2d}$ to a symplectic isotopy $(\Phi_t)
$ on $T^*\T^d$ such that for every $t\in\R$
\[
\Pi\circ \Phi_t=\phi_t\circ\Pi\, .
\]
Let $\cz_0\subset T^*\T^d$ be the zero section, which is a Lagrangian submanifold. By Theorem \ref{Tppal} for every $n\in\N$ there exists a point $u_n\in\cz_0$ such that
\[
\mathrm{DMI}(T_{u_n}\cz_0, (\Phi_t)_{t\in[0,n]})=0\, .
\]
Since the covering $\Pi$ is a submersion, for every $L\in\Lambda(T^*\T^d)$ we have
\[
\mathrm{DMI}(D\Pi(L), (\phi_t)_{t\in[0,1]})=\mathrm{DMI}(L,(\Phi_t)_{t\in[0,1]})\, ,
\]
where the Maslov index in $T^*\T^d$ is calculated with respect to the vertical Lagrangian foliation $\cv$ whose associated tangent bundle is $T_xT^*\T^d$, while the Maslov index in $\T^{2d}$ is calculated with respect to the image foliation $\Pi(\cv)$. Observe that the tangent bundle associated to $\Pi(\cv)$ is $\mathrm{ker}\,(dp_1)$, where $p_1:\T^{2d}\to \T^d$ is the projection of the first $d$-coordinates.

For every $n\in\N$ define then $U_n:=D\Pi(T_{u_n}\cz_0)\in\Lambda(\T^{2d})$ and the probability measure on $\Lambda(\T^{2d})$
\[
\mu_n:=\frac 1 n \sum_{i=0}^{n-1}\delta_{D\phi_i(U_n)}\, .
\] 
Since $\Lambda(\T^{2d})$ is compact, we can extract from $(\mu_n)_{n\in\N}$ a subsequence converging to a $D\phi_1$-invariant probability measure $\bar\mu$ on $\Lambda(\T^{2d})$. Using the projection $p:\Lambda(\T^{2d})\to\T^{2d}$ and repeating the calculus done in the proof of Corollary \ref{cormeasurebis}, the $\phi_1$-invariant probability measure { $\mu=p_*\bar\mu$} on $\T^{2d}$ is then such that
\[
\mathrm{DMI}(\mu,(\phi_t)
)=0\, .
\]

\subsection{Existence of points and ergodic measures with vanishing asymptotic Maslov index for conformally symplectic isotopies that twist the vertical}
In this subsection we are mainly concerned with the proof of Theorems \ref{existencepoints0DMIbis} and \ref{existencepoints0DMItris}. Let us first recall that, in Proposition \ref{PtwistMAslov}, we prove that, for an isotopy $(\phi_t)_{t\in\R}$ of conformally symplectic diffeomorphisms of $\M$ that twists the vertical, for every $L\in\Lambda(\M)$ and every $[\alpha,\beta]\subset\R$ such that $D\phi_\alpha(L),D\phi_\beta(L)\notin\Sigma(\M)$ it holds $$\mathrm{DMI}(L,(\phi_t)_{t\in[\alpha,\beta]})\leq 0\, .$$
Consequently, for every $x\in \M$ we have
\[
\mathrm{DMI}_\infty(x,(\phi_t)_{t\in[0,+\infty)})\leq 0\, .
\]
Moreover, from Corollary \ref{comparison DAMI et DMI}, we deduce that, for an isotopy $(\phi_t)_{t\in\R}$ of conformally symplectic diffeomorphisms on $\M=T^*M$ that twists the vertical, for every $L\in\Lambda$ and every $x\in \M$ it holds
\[
\mathrm{D}\alpha\mathrm{MI}(L, (\phi_t)_{t\in[0,1]})< d\quad\text{and}\quad \mathrm{D}\alpha\mathrm{MI}_\infty(x,(\phi_t)_{t\in[0,+\infty)})\leq 0\, ,
\]
where $d=\mathrm{dim}(M)$.


\begin{proof}[Proof of Theorem \ref{existencepoints0DMIbis}]
Let $(\phi_t)_{t\in\R}$ be a conformally symplectic isotopy of ${\cal M}$ such that $\phi_0=\mathrm{Id}_{\cal M}$. Let $\cl\subset {\cal M} $ be a Lagrangian submanifold   that is Hamiltonianly isotopic to a graph.
Let $\cl_0\subset\M$ be a Lagrangian graph and let $(h_t)_{t\in[0,1]}$ be a Hamiltonian isotopy such that $h_0=\mathrm{Id}_{\M}$ and $h_1(\cl_0)=\cl$. Let $\alpha:[0,1]\to[0,1]$ be a smooth non-decreasing function such that $\alpha(0)=0$ and $\alpha$ is constant equal to $1$ when restricted to some neighborhood of $1$. {Let $\beta:[0,1]\to[0,1]$ be a smooth non-decreasing function such that $\beta$ is constant equal to $0$ on some neighborhood of $0$ and equal to the identity on some neighborhood of $1$.} Define then $(\psi_t)_{t\in[0,+\infty)}$ as
\[
{\psi_t:=\begin{cases}
h_{\alpha(t)}\quad \text{for }t\in\left[0,1\right]\, , \\
\\
\phi_{\beta(t-1)}\circ h_1\quad\text{for }t\in[1,2]\, , \\
\\
\phi_{t-2}\circ h_1\quad \text{for }t\in[2,+\infty)\, .
\end{cases}}
\]
Then $(\psi_t)_{t\in[0,+\infty)}$ is an isotopy of conformally symplectic diffeomorphisms such that $\psi_0=\mathrm{Id}_\M$.	

Apply then Theorem \ref{Tppal} to the Lagrangian graph $\cl_0$ with respect to the isotopy $(\psi_t)_{t\in[0,+\infty)}$. That is, for every $t\in[0,+\infty)$ there exists at least a point $z_t\in \cl_0$\footnote{Actually there exists an open set whose projection on $M$ has full Lebesgue measure.} such that
\begin{equation}\label{0 value}
\mathrm{DMI}(T_{z_t}\cl_0,(\psi_s)_{s\in[0,t]})=0\, .
\end{equation}
From \eqref{0 value} and from the compactness of $\{ h_s(\cl_0) :\ s\in[0,1] \}$, there exists an integer $\rho>0$ such that for every $t\in[0,+\infty)$ there exists a point $x_t:=\psi_1(z_{t+1})=h_1(z_{t+1})\in \cl$ such that
$$
\mathrm{DMI}(T_{x_t}\cl,(\phi_s)_{s\in[0,t]})\in[-\rho,\rho]\, .
$$
{Moreover, as $(\phi_s)$ twists the vertical, we have in fact

\begin{equation}\label{choice of zt}
\mathrm{DMI}(T_{x_t}\cl,(\phi_s)_{s\in[0,t]})\in[-\rho,0]\, .
\end{equation}
}
By compactness of $\cl$, we can extract from $(x_t)_{t\in[0,+\infty)}$ a subsequence $(x_n)_{n\in\N}$ which converges to a point $ x\in\cl$.

Fix $N\in\N$ and $\epsilon>0$. By continuity of the angular Maslov index, there exists $\bar n\in\N$ such that for every $n\geq\bar n$ it holds
\begin{equation}\label{continuity DAMI}
\Big\vert \mathrm{D}\alpha\mathrm{MI}(T_{ x}\cl, (\phi_s)_{s\in[0,N]}) - \mathrm{D}\alpha\mathrm{MI}(T_{x_n}\cl, (\phi_s)_{s\in[0,N]}) \Big\vert<\epsilon\, .
\end{equation}
Since the isotopy twists the vertical, we claim that, for every $n\geq \max (\bar n, N)$, it holds
\begin{equation}\label{control aussi avant}
\mathrm{DMI}(T_{x_n}\cl,(\phi_s)_{s\in[0,N]})\in[-\rho, 0]\, .
\end{equation}
Indeed, if this does not hold, then for some $n\geq \max(\bar n, N)$ from Proposition \ref{PtwistMAslov} we have that
\[
\mathrm{DMI}(T_{x_n}\cl,(\phi_s)_{[0,N]})\leq -\rho-1\, .
\]
From \eqref{choice of zt} and since
\[
\mathrm{DMI}(T_{x_n}\cl,(\phi_s)_{[0,n]})=\mathrm{DMI}(T_{x_n}\cl,(\phi_s)_{[0,N]})+\mathrm{DMI}(D\phi_N(T_{x_n}\cl),(\phi_s)_{s\in[0,n-N]})\, ,
\]
we contradict Proposition \ref{PtwistMAslov} because
\[\mathrm{DMI}(T_{\phi_N(x_n)}\phi_N(\cl),(\phi_s)_{s\in[0,n-N]})\geq 1\, .\]
From \eqref{continuity DAMI}, \eqref{control aussi avant} and Corollary \ref{comparison DAMI et DMI}, we have that for every $n\geq\max(\bar n, N)$
\[\begin{split}
\Big\vert \mathrm{D}&\alpha\mathrm{MI}(T_{x}\cl,(\phi_s)_{s\in[0,N]}) \Big\vert\\
&\leq\Big\vert \mathrm{D}\alpha\mathrm{MI}(T_{ x}\cl,(\phi_s)_{s\in[0,N]}) -\mathrm{D}\alpha\mathrm{MI}(T_{x_n}\cl, (\phi_s)_{s\in[0,N]})\Big\vert+ \Big\vert \mathrm{D}\alpha\mathrm{MI}(T_{x_n}\cl, (\phi_s)_{s\in[0,N]})\Big\vert\\
&<\epsilon+\rho+d\, ,
\end{split}
\]
where $d=\mathrm{dim}(M)$. Letting $\epsilon\to 0$ and again by Corollary \ref{comparison DAMI et DMI}, for every $t\in[0,+\infty)$ we conclude that
\[
\mathrm{DMI}(T_x\cl,(\phi_s)_{s\in[0,t]})\in[-C,C]\, ,
\]
where $C:= \rho+2d$. In particular, we deduce also that $\mathrm{DMI}_\infty(x,(\phi_t)_{t\in[0,+\infty)})=0$.

\end{proof}

\begin{proof}[Proof of Theorem \ref{existencepoints0DMItris}]
{	Let $(\phi_t)$ be an isotopy of conformally symplectic diffeomorphisms of $\M$ such that $\phi_{1+t}=\phi_t\circ\phi_1$. Observe that, if $(\phi_t)$ twists the vertical, then, by Proposition \ref{PtwistMAslov}, for every 
$\phi_1$-invariant measure with compact support $\mu$ it holds
	\begin{equation}\label{eq mes twist}
	\mathrm{DMI}(\mu,(\phi_t)
	)=\int_\M \mathrm{DMI}_\infty(x,(\phi_t)
	)d\mu(x) \leq 0\, .
	\end{equation}
	{ As the function $\mathrm{DMI}(.,(\phi_t)
	)$ is measurable and non-positive, this implies that
	$\mathrm{DMI}(.,(\phi_t)_{t\in[0,+\infty)})\in L^1(\mu)$.}\\
Let $x\in\cl$ be the point given by Theorem \ref{existencepoints0DMIbis}. The assumption that its positive orbit is relatively compact enables us to find a $\phi_1$-invariant measure $\mu$ supported on the closure of the orbit of $x$ with vanishing asymptotic Maslov index.\\
By Ergodic Decomposition Theorem (see \cite{Mane1987}), {for $\mu$ almost every $y$, the measure 
$$\mu_y=\lim_{N\to\infty}\frac{1}{N}\sum_{n=0}^N\delta_{\phi_n(y)}$$ exists and is ergodic, we have $\mathrm{DMI}(.,(\phi_t)
)\in L^1(\mu_y)$ and 
\[
0=\mathrm{DMI}(\mu,(\phi_t)
)=\int_{\M}\mathrm{DMI}(\mu_y,(\phi_t)
)d\mu(y)\, .
\] 
As the function in the integral is non-positive by \eqref{eq mes twist}, we deduce that for $\mu$ almost every $y$, the measure $\mu_y$ is ergodic and has vanishing Maslov index.}

}
\end{proof}

\subsection{Autonomous and 1-periodic Tonelli Hamiltonian flow case}
We can consider the particular case of a Hamiltonian 1-periodic Tonelli flow on a cotangent bundle $T^*M$, where $M$ is a $d$-dimensional compact manifold. More precisely, let $H:T^*M\times \R/\Z\to \R$ be a Tonelli $1$-periodic Hamiltonian. Denote as $(\phi^H_{s,t})$ the family of symplectic maps generated by the Hamiltonian vector field of $H$.

Using Weak KAM Theory, we can easily obtain Theorem \ref{existencepoints0DMIbis} for a Lagrangian graph. 
More precisely, let $\cl\subset T^*M$ be a Lagrangian graph, that is there exists a $C^{1,1}$ function $u:M\to \R$ such that $\cl=\mathrm{graph}\,du$. Then, the existence of a point $x\in\cl$ with zero asymptotic Maslov index can be deduced from Weak KAM theory. Indeed, let $v:M\to\R$ be a Weak KAM solution of positive type. In particular, $v$ is semiconvex.
Consider then the function $v-u$, which is still semiconvex. Let $x_0\in M$ be a local maximum of the function $v-u$. Then, since $v-u$ is semiconvex and $x_0$ is a local maximum, actually the function $v-u$ is differentiable at $x_0$. We deduce that $dv(x_0)=du(x_0)$. Consequently, the Lagrangian submanifold $\cl$ intersects the partial graph of $dv$ in $du(x_0)$. Since $v$ is a weak KAM solution of positive type, the orbit of a point lying in the partial graph of $dv$ is minimizing on every interval $[0,t]$, for $t>0$. In particular, the point $du(x_0)$ does not have conjugate points in the future. This implies that the Maslov index on every interval $[0,t]$ at $du(x_0)$ is zero, and so $du(x_0)\in\cl$ has zero asymptotic Maslov index.

Recall that an autonomous Tonelli Hamiltonian flow provides an isotopy of symplectic diffeomorphisms that twists the vertical, see Proposition \ref{PHconvextwist}. By Theorem \ref{existencepoints0DMItris} there exists then an ergodic invariant measure of vanishing asymptotic Maslov index. In the case of an autonomous Tonelli Hamiltonian flow on $T^*M$, we can characterise the invariant measure of vanishing Maslov index given by Theorem \ref{existencepoints0DMItris}. That is, the given invariant measure is actually a Mather minimizing measure, as stated in Corollary \ref{cor on minimizing measure}.
\begin{proof}[Proof of Corollary \ref{cor on minimizing measure}]
	Indeed it can be proved that the graph selector is unique (see e.g. \cite{ArnaVent2017}). In this case, the graph selector can be built by using the Lax-Oleinik semi-group, {\color{black}see \cite{Jou91} or \cite{Wei14}}. Fathi, \cite{Fathi2008}, proved the convergence of the Lax-Oleinik semi-group (that is the graph selectors in our case)  to a weak KAM solution. Arnaud proved  in \cite{Arnaud2005} that the resulting pseudographs converge to the pseudograph of a weak KAM solution for the Hausdorff distance. Hence the supports of measures  that are given by the  last theorem are in the pseudograph of a weak KAM solution and thus minimizing, { see (3.14) in \cite{Bernard08}}.
\end{proof}

\subsection{Proof of Corollary \ref{hausdorff dimension}} { We endow $M$ with a Riemannian metric. 

} 
{
	We are assuming that \begin{equation}\label{cor dim hau hyp}\forall (\lambda_1,\dots,\lambda_n)\in\R^n\setminus\{0_{\R^n}\}, \forall q\in M\quad\text{it holds}\quad \sum_{k=1}^n\lambda_k\eta_k(q)\neq 0\, .\end{equation}
{ This implies that the  map $I$
from $M\times \R^n$ to $\M=T^*M$ that is defined by
$$I(q, \lambda_1, \dots, \lambda_n)=\sum_{k=1}^n\lambda_k\eta_k(q)$$ is a bi-Lipschitz embedding. Indeed, it is a fibered linear monomorphism  from $M\times \R^n$ to $T^*M$ that continuously  depends on the point $q\in M$. We denote by $\cal Q\subset \M$ its image $I(M\times \R^n)$. Then the map $j:\cal Q\to\R^n$ that is defined by $$j\big(\sum_{k=1}^n\lambda_k\eta_k(q)\big)=(\lambda_1, \dots, \lambda_n)$$ is Lipschitz.}

For every $(\lambda_1,\dots,\lambda_n)\in\R^n$ we consider the Lagrangian graph $\cl_{(\lambda_1,\dots,\lambda_n)}:=\{
	\sum_{k=1}^n\lambda_k\eta_k(q) :\ q\in M\}\subset \M$. As $(\phi_t)$ is an isotopy of conformally symplectic diffeomorphisms that twists the vertical, from Theorem \ref{existencepoints0DMIbis}, there exists at least one point $x\in\cl_{(\lambda_1,\dots,\lambda_n)}$ with zero asymptotic Maslov index. In particular,{ 
	\[
	 j\left( \{ p\in\cal Q :\ \mathrm{DMI}_\infty(p,(\phi_t)_{t\in[0,+\infty)})=0 \} \right)= \R^n\, .
	\]
	Because $j$ is Lipschitz, this implies that}
	\[
	\mathrm{dim}_H\Big(\{ p\in \M :\ \mathrm{DMI}_\infty(p,(\phi_t)_{t\in[0,+\infty)})=0 \}\Big)\geq n\, .
	\]
	

	



}

\bibliographystyle{alpha}
\bibliography{Biblio.bib}

\end{document}